\documentclass[preprint,nopreprintline,12pt]{elsarticle}

\usepackage{lineno}
\usepackage{babel}
\usepackage{ulem}

\journal{J. Comput. Phys.}

\usepackage{amsmath,amssymb,amsthm,mathtools}
\usepackage{dsfont}
\usepackage{todonotes}
\usepackage{ifthen}
\usepackage{bm}
\usepackage{xparse}
\usepackage{xspace}
\usepackage{xstring}
\usepackage{natbib}

\usepackage{tikz}
\usepackage{tikzscale}
\usepackage{pgfplots}
\pgfplotsset{compat=newest}
\usepackage{pgfplotstable}
\usepgfplotslibrary{groupplots,dateplot}

\usepackage{hyperref}
\hypersetup{
	colorlinks,
	linkcolor={red!50!black},
	citecolor={blue!50!black},
	urlcolor={blue!80!black}
}
\usepackage[nameinlink]{cleveref}
\usepackage{autonum}
\usepackage{enumitem}
\setenumerate{label=\upshape(\alph*),itemsep=3pt,topsep=3pt}

\usepackage{KITcolors}

\newcommand{\GWP}{Gaussian wave packet\xspace}
\newcommand{\GWPs}{Gaussian wave packets\xspace}

\newcommand{\Schroed}{Schr{\"o}\-din\-ger\xspace}
\newcommand{\secondorder}{second-order\xspace}
\newcommand{\semiclassical}{semiclassical\xspace}
\newcommand{\timedependent}{time-dependent\xspace}

\newcommand{\divfree}{divergence-free\xspace}
\newcommand{\wavepacket}{wave packet\xspace}
\newcommand{\timestep}{time step\xspace}

\newcommand{\EulerLagrange}{Euler-Lagrange\xspace}
\newcommand{\electroMagnetic}{electro-magnetic\xspace}

\newcommand{\tr}{\mathrm{tr}\:}

\renewcommand{\i}{\mathrm{i}}

\newcommand{\bbR}{\mathbb{R}}
\newcommand{\Rd}{\bbR^\dim}

\newcommand{\bbC}{\mathbb{C}}
\newcommand{\da}{\coloneqq}

\newcommand{\Id}{\mathrm{Id}}
\newcommand{\sol}{\psi}
\newcommand{\vsol}[1][]{{u^{#1}}}

\newcommand{\pos}[1][]{{q^{#1}}}
\newcommand{\mom}[1][]{{p^{#1}}}
\newcommand{\wm}[1][]{{\mathcal{C}^{#1}}}
\newcommand{\fcP}[1][]{{P^{#1}}}

\newcommand{\fcQ}[1][]{{Q^{#1}}}
\newcommand{\pha}[1][]{\zeta^{#1}}

\newcommand{\vel}[1][]{v^{#1}}

\newcommand{\fcVel}[1][]{\Upsilon^{#1}}
\newcommand{\fcVelExtr}[1][]{\Upsilon^{#1}}

\newcommand{\RC}[1][]{{{\wm}^{#1}_{\textup{R}}}}

\newcommand{\IC}[1][]{{{\wm}_\textup{I}^{#1}}}
\newcommand{\ICinv}{{\wm_\textup{I}^{-1}}}

\newcommand{\ptd}{\delta}

\newcommand{\ReC}{{\wm_\textup{R}}}
\newcommand{\ImC}{{\wm_\textup{I}}}
\newcommand{\ImCinv}{{{\wm}_\textup{I}^{-1}}}

\newcommand{\RePha}[1][]{\pha[#1]_{\textup{R}}}
\newcommand{\ImPha}[1][]{\pha[#1]_{\textup{I}}}

\newcommand{\somePot}{W}
\newcommand{\somePots}{w}
\newcommand{\mgPot}[1][]{{A^{#1}}}
\newcommand{\elField}[1][]{{{E}^{#1}}}
\newcommand{\mgField}[1][]{{{B}^{#1}}}
\newcommand{\mgFieldStrength}{B_0}
\newcommand{\effmgFieldStrength}{B_m}

\newcommand{\mltPot}{\phi}

\newcommand{\magnetFreq}{\omega_{-}}
\newcommand{\corrCycloFreq}{\omega_{+}}
\newcommand{\magnAndCycloFreq}{\omega_{\pm}}
\newcommand{\axialFreq}{\omega_3}
\newcommand{\cycloFreq}{\omega_c}

\newcommand{\magnetFreqPi}{\nu_{-}}
\newcommand{\corrCycloFreqPi}{\nu_{+}}

\newcommand{\axialFreqPi}{\nu_3}

\newcommand{\charge}{q_e}

\newcommand{\someMatrix}{M}

\newcommand{\auxPot}[1][]{{E^{#1}}}
\newcommand{\auxauxPot}[1][]{{S^{#1}}}

\NewDocumentCommand{\mean}{O{} m}{#1\langle #2 #1\rangle_{\vsol}}

\newcommand{\Mf}{\mathcal{M}}

\newcommand{\obs}{\mathbf{A}}

\newcommand{\op}{\mathrm{op}_\mathrm{Weyl}}
\newcommand{\re}{\mathrm{Re}\,}
\newcommand{\im}{\mathrm{Im}\,}
\renewcommand{\l}{\left}

\newcommand{\Ham}{H}

\newcommand{\scp}{\varepsilon}
\newcommand{\cham}{h}

\renewcommand{\dim}{d}

\newcommand{\ii}{\mathrm{i}}

\newcommand{\Div}{\mathrm{div}}

\newcommand{\curl}{\nabla\times}
\newcommand{\sign}{\operatorname{sign}}

\newcommand{\diff}{\mathop{}\!\mathrm{d}}

\newcommand{\dx}{\diff x}

\newcommand{\intRd}{\int_{\Rd}}

\newcommand{\pt}{\partial_{t}}

\newcommand{\jacobian}[1]{J_{#1}}

\newcommand{\abs}[1]{\left\lvert #1 \right\rvert}

\newcommand{\Lin}[2][]{\mathcal{L}(\ifthenelse{\equal{#1}{}}{#2}{#1,#2})}

\newcommand{\tn}[1][n]{t^{#1}}

\newtheorem{theorem}{Theorem}[section]
\newtheorem{lemma}[theorem]{Lemma}

\newtheorem{prop}[theorem]{Proposition}

\newtheoremstyle{break}
  {\topsep}{\topsep}%
  {\itshape}{}%
  {\bfseries}{}%
  {\newline}{}%
\theoremstyle{break}
\newtheorem{algo}[theorem]{Algorithm}

\theoremstyle{definition}

\theoremstyle{remark}
\newtheorem{remark}[theorem]{Remark}

\numberwithin{equation}{section}

\newcommand{\norm}[2]{\left\lVert #2 \right\rVert_{#1}}

\newcommand{\normLp}[2]{\norm{L^{#1}}{#2}}

\newcommand{\normLtwo}[1]{\normLp{2}{#1}}

\newcommand{\Ltwo}{L^2}

\begin{document}

\begin{frontmatter}



\title{Time-integration of Gaussian variational approximation for the magnetic Schr\"odinger equation}

\author[1]{Malik Scheifinger\corref{cor1}} 
\ead{malik.scheifinger@kit.edu}

\author[2,3]{Kurt Busch} 
\ead{kurt.busch@physik.hu-berlin.de}

\author[1]{Marlis Hochbruck} 
\ead{marlis.hochbruck@kit.edu}

\author[4]{Caroline Lasser} 
\ead{classer@tum.de}

\cortext[cor1]{Corresponding author}

\affiliation[1]{%
            organization = {Institute for Applied and Numerical Mathematics, Karlsruhe Institute of Technology},
            addressline={Englerstr.~2}, 
            postcode={76131},
            postcodesep={},
            city={Karlsruhe},  
            country={Germany}}

\affiliation[2]{%
            organization = {Institut f\"ur Physik, Humboldt Universit\"at zu Berlin},
            addressline={Newtonstr.~15}, 
            postcode={12489},
            postcodesep={},
            city={Berlin},  
            country={Germany}}   
            
\affiliation[3]{%
            organization = {Max-Born-Institut},
            addressline={Max-Born-Str.~2A}, 
            postcode={12489},
            postcodesep={},
            city={Berlin},  
            country={Germany}}   
            
\affiliation[4]{%
            organization = {Department of Mathematics, Technische Universit\"at M\"unchen},
            addressline={Boltzmannstr.~3}, 
            postcode={85748},
            postcodesep={},
            city={Garching b.~M\"unchen},  
            country={Germany}} 

\begin{abstract}
In the present paper we consider the semiclassical magnetic 
\Schroed
equation, which describes the dynamics of charged particles under the influence of a electro-magnetic field. The solution of the
 \timedependent 
Schr{\"o}dinger equation is approximated by a 
single Gaussian wave packet via the time-dependent Dirac--Frenkel variational principle.
For the approximation we use ordinary differential equations of motion for the parameters of the variational solution and extend the 
\secondorder 
Boris algorithm for classical mechanics to the quantum mechanical case. In addition, we propose a modified version of the classical fourth order Runge--Kutta method. Numerical experiments explore parameter convergence and geometric properties. Moreover, we benchmark against 
the analytical solution of the Penning trap.  
\end{abstract}





\end{frontmatter}


\section{Introduction}
\label{sec:intro}


In the present paper we study the numerical time-integration for 
charged quantum particles that are
subjected to external magnetic and electric fields. The dynamics is governed by  
the \semiclassical magnetic Schrödinger equation 
\begin{subequations} \label{eq:sproblem-all}
	\begin{equation+} \label{eq:sproblem}
		\ii\scp\pt \sol(t) =\Ham(t)\sol(t),\quad \sol(0) = \sol_0, \quad t\in \bbR,
	\end{equation+}
	on $\bbR^\dim$ with magnetic Hamiltonian 
	\begin{equation+}\label{eq:ham}
		\Ham(t)=  \frac12 \bigl(-\ii\scp \nabla_x -
        \mgPot(t,\cdot)\bigr)^2 
        + 
        \mltPot(t,\cdot),
	\end{equation+}
\end{subequations}
and initial value $\sol_0 \in L^2(\bbR^\dim)$ with \semiclassical parameter $0< \scp \ll 1$. Here, $\mgPot$ is a \divfree magnetic vector potential, and $\mltPot$ is the electric potential. 
%
From a numerical point of view, solving this time-dependent partial differential equation raises three major problems. First, it is a high-dimensional problem, since the space dimension is typically given by $\dim = 3N$, where 
$N$ is the number of quantum particles in the system. Further, the computational domain~$\bbR^\dim$ is naturally  unbounded, and thus most numerical methods require truncation before discretization. For the method of lines (first discretize space, then time), high dimension combined with an unbounded domain leads to inadequately if not intractably large systems that have to be integrated in time. 
Another challenge is given by the high oscillations induced by the small \semiclassical parameter~$\scp$. For standard time integration schemes severe step-size restrictions have to be imposed and leave these methods impracticable. 

We consider the case that the initial condition $\sol_0$ is strongly localized and given by a Gaussian \wavepacket, and investigate the numerical time-integration of an approximate solution $\vsol(t)\approx\sol(t)$, which is a \GWP 
\[
\vsol(t,x) = \exp{\Bigl(\frac{\ii}{\scp}\bigl(\frac12(x-\pos_t)^\top \wm_t (x-\pos_t)+(x-\pos_t)^\top \mom_t +\pha_t\bigr)\Bigr)}.
\]
The parameters to be computed are the packet's position and momentum center $\pos_t,\mom_t$, the complex width matrix $\wm_t$ of the envelope, and the complex phase and weight parameter $\pha_t$. These parameters evolve according to ordinary differential equations, which are systematically derived by the Dirac--Frenkel time-dependent variational principle. 
By the approximation ansatz, high oscillations in time and space are captured and thus eliminated for the numerical time integration.  
For $\mgPot = 0$ it is well established that variational \GWPs offer reasonable mesh-free approximations at low 
computational cost, see for example \cite{LasL20,Van23}. More recently, they have also been proposed for magnetic quantum dynamics \cite{BurDHL24}.

\subsection{Contributions of the paper}
As our main contribution we derive two fast algorithms to solve the equations of motion for the parameters of a variational \GWP approximation such that norm and energy conservation of the quantum solution are echoed by the time integrator. 
On a standard laptop with a non-optimized Jupyter Notebook\footnote{Codes available at
\url{https://gitlab.kit.edu/malik.scheifinger/magnschroedti}.}, these algorithms enable us to compute approximations within minutes. 
First, we suggest an extension of the Boris algorithm \cite{Boris70, Bir18}, which is a standard integrator for the classical equations of motion for a charged particle system in plasma physics, to the quantum mechanical setting. We furthermore modify the well-known Runge--Kutta 4 method, such that it conserves the \(L^2\)-norm of a \GWP at every \timestep. Numerical experiments in two and three space dimensions underline the efficiency and expected 
accuracy of the proposed algorithms.

%
%

\subsection{Outline of the paper}
The paper is structured as follows. In \Cref{subsec:intro-Penning-trap}, we discuss the quantum dynamics of an electron and a proton in a 
hyperbolic Penning trap as a guiding example for a magnetic Schr\"odinger equation in the \semiclassical regime. 
In \Cref{sec:num-approx}, we compare variational and traditional \GWP approximation and review the known error 
estimates.
In \Cref{sec:teom}, we transform the system of ordinary differential equations, that determines the variational parameter evolution, in a form that features averaged magnetic momenta and contains the magnetic vector field $\mgField = \curl\mgPot$ on the right hand side. In particular, we derive an equation for the imaginary part of the phase parameter $\pha$ guaranteeing the preservation of the \(L^2\)-norm. In \Cref{sec:time-int}, we briefly introduce the Boris algorithm used and derive our two algorithms for solving the parameters ODEs for the approximating \GWP. Finally, in \Cref{sec:experiments}, we present numerical experiments for a two-dimensional magnetic system with trigonometric potentials and for the three-dimensional Penning trap. \Cref{appendix} gives formulas for the magnetic energy and other averages of \GWPs.

\subsection{Notation}

%
%
%
For a scalar function $a\colon\bbR^{\dim}\to\bbR$ we denote the Hessian matrix by $\nabla^2 a(x)$ and
for a vector field $A\colon \bbR^\dim\to\bbR^\dim$ we denote the Jacobi matrix by 
$J_A(x) = (\partial_\ell A_k(x))_{k,\ell =1}^\dim$.  
For a function $W \colon \bbR^\dim \to \bbR^L$, $L\geq 1$, and more generally a linear operator $\obs$ acting on $L^2(\bbR^d)$
we define the averages
\begin{align}
	\langle W\rangle_{\vsol} \coloneqq \langle \vsol |W\vsol\rangle, \quad
\langle \obs\rangle_{\vsol} &\coloneqq \bigl\langle \vsol| \obs \vsol \bigr\rangle, 
\end{align}
if the inner products exist. We follow the convention that inner products are anti-linear in the first component. We also use the dot product of vectors  
$v,w\in\bbC^L$ as $v\cdot w \coloneqq v^\top w = v_1w_1 + \cdots + v_L w_L$ and the square $v^2 \coloneqq v\cdot v$. For complex matrices $\wm\in\bbC^{\dim\times\dim}$ we denote the 
component-wise real and imaginary parts by $\RC,\IC\in\bbR^{\dim\times\dim}$, respectively.

\section{Penning trap}\label{subsec:intro-Penning-trap}

To illustrate semiclassical scaling for magnetic quantum dynamics we consider a charged
microscopic particle, e.g., an electron or a proton, in a macroscopic hyperbolic 
Penning trap.
Such a trap consists of an arrangement of magnetic coils and electrodes 
which feature a static magnetic field along the  $x_3$-direction and a quadrupole-like static electric field that is rotationally symmetric 
around the $x_3$-axis, see, e.g., \cite{BroL84}. The corresponding vector and scalar potentials, respectively, read as
\begin{align}
	\label{eq:trap-potentials}
	\mgPot(x) = \frac{1}{2} \mgFieldStrength
       \begin{pmatrix}
          -x_2 \\
           x_1 \\
           0
       \end{pmatrix}
	\quad \text{and} \quad
	\mltPot(x) = \frac{ \mltPot_0}{2\ptd^2} 
	    \left( x_3^2 - \frac{1}{2} \left( x_1^2 + x_2^2 \right) \right),
\end{align} 
see \Cref{tab:trap-frequencies} for typical trap parameters $\mgFieldStrength, \mltPot_0, \ptd$ for protons and electrons.
A classical point particle (mass $m$, charge $\charge$) moving in such an  
electromagnetic field configuration executes an oscillatory motion 
along the $x_3$-axis (angular frequency~$\axialFreq$), while simultaneously executing an epicyclic motion in the $x_1x_2$-plane where the low-frequency magnetron orbit (magnetron frequency~$\magnetFreq$) is overlayed with 
high-frequency cyclotron orbits (angular frequency~$\corrCycloFreq$), cf.\ \Cref{fig:penning-trap-projected-exact}.
In terms of the particle and trap parameters, these frequencies are 
given by
\begin{align}
	\magnAndCycloFreq = \frac{1}{2} \left( \cycloFreq \pm \Omega \right), 
    \quad
	\axialFreq = \sqrt{\frac{\vert \charge \vert \mltPot_0}{m\ptd^2}},
	\quad \text{where} \quad 
	\cycloFreq = \frac{\vert \charge \vert \mgFieldStrength}{m}, \quad
    \Omega = \sqrt{\cycloFreq^2 - 2 \axialFreq^2}.
\end{align}
The associated classical trajectories $x_c(t)$
may be obtained via Newtonian~\cite{BroL84} or Hamiltonian~\cite{Kre2014} 
approaches. 

\begin{table}
	\begin{center}
    \renewcommand{\arraystretch}{1.3}
    \small{
		\begin{tabular}{|l|r|r|}
			\hline  
			\textbf{Quantity} & \textbf{Electron} & \textbf{Proton} \\
			\hline 
		$\ptd$   (Trap Size)           &      0.00335 m    &       0.00112 m \\
		$\mgFieldStrength$   (Magnetic Field)      &        5.872 T    &         5.050 T \\
		$ \mltPot_0$ (Electrode Potential) &        10.22 V    &         53.10 V \\          
			$\corrCycloFreqPi = \frac{\corrCycloFreq}{2 \pi}$ (Corrected Cyclotron Frequency) & 164.38 GHz & 76.299 MHz\\
			$\axialFreqPi = \frac{\axialFreq}{2 \pi}$ (Axial Frequency)               & 63.698 MHz & 10.134 MHz \\
			$\magnetFreqPi = \frac{\magnetFreq}{2 \pi}$ (Magnetron Frequency)           & 12.341 kHz & 672.93 kHz \\
			\hline
		\end{tabular}
        }
	\end{center}
	\caption{\label{tab:trap-frequencies}
		     Typical frequencies for the undulatory motion of electrons and
		     protons in a Penning trap, cf.\ \cite[Tables~I and II]{BroL84}.}
\end{table}

\begin{figure}[!htb]
    \centering
    \includegraphics[width=.7\linewidth]{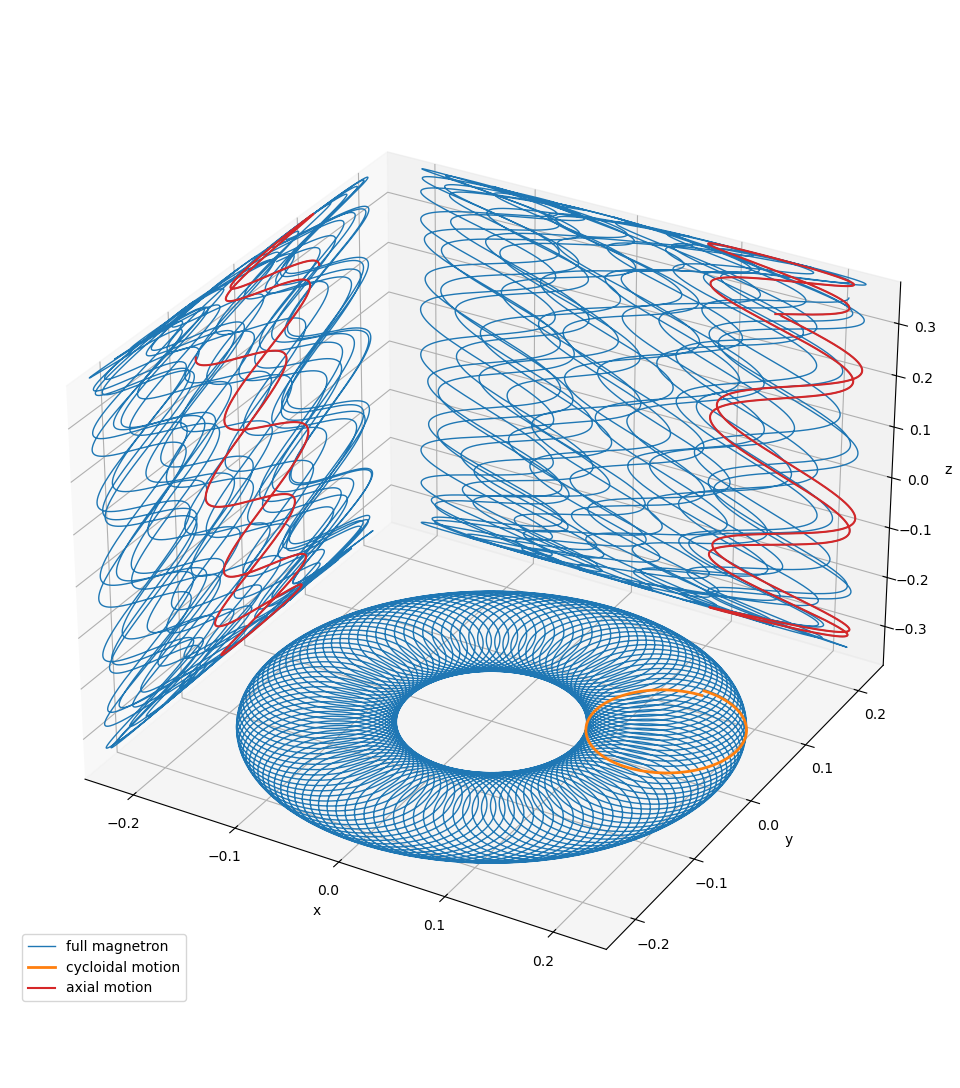}
    \caption{Exact trajectory of a proton in a hyperbolic Penning trap with data from \Cref{tab:trap-frequencies} and initial condition specified in \Cref{subsec:Penning-trap} on the dimensionless time interval \([0,2\pi]\).}
    \label{fig:penning-trap-projected-exact}
\end{figure}
The quantum dynamics of a trapped particle is governed by the 
time-dependent Schr\"{o}dinger equation which, in SI units, reads
\begin{align}
	\label{eq:TDSE-EM}
	\ii \hbar \, \partial_t \sol (t,x) 
	=
	\Bigl(
	\frac{1}{2m}
	\bigl(  -\ii \hbar\nabla_x - \charge \mgPot(x) \bigr)^2 +  \charge \mltPot(x)
	\Bigr) \sol(t,x).
\end{align}
%
A first comparison of the macroscopic spatial extent of the trap, described by the trap parameter $\ptd \approx 1$mm,  with the wavelength of the trapped quantum particle (e.g., a proton) corresponding to the corrected cyclotron frequency $\corrCycloFreqPi \approx 76$MHz indicates already that mesh-based approaches to \eqref{eq:TDSE-EM} are rather impracticable.
Upon transiting to dimensionless coordinates
$x \to x/\ptd$ 
and introducing the dimensionless time 
$t \to \magnetFreq t$,
equation \eqref{eq:TDSE-EM} becomes
\begin{align}
	\label{eq:TDSE-no-units}
	\ii \scp \partial_{t} \sol (t,x)
	=
	\left(
	\frac{1}{2} \left( \ii \scp {\nabla} 
	+  \mgPot_m(x)
	\right)^2 
	+
    \sign(\charge)
	\frac{\corrCycloFreq}{\magnetFreq} 
	\Bigl( 
        x_3^2 - \frac{1}{2} 
        \left( 
            x_1^2 + x_2^2 
	    \right) 
	\Bigr) 
	\right) \sol (t,x). 
\end{align}
This is of the form \eqref{eq:sproblem-all}, since 
with the effective magnetron magnetic field 
$\effmgFieldStrength = m \magnetFreq/  \charge $ we obtain the scaled 
dimensionless quantities
\begin{align}
& \scp = \hbar/(\charge \effmgFieldStrength \ptd^2) \approx 1.19\cdot 10^{-8},
\quad {\corrCycloFreq}/{\magnetFreq} \approx 113.25,\\
&	\mgPot_m (x)  = \frac{1}{2} \frac{\mgFieldStrength}{\effmgFieldStrength} 
	\begin{pmatrix}
		- x_2 \\
		x_1 \\
		0
	\end{pmatrix},
    \quad  \frac{\mgFieldStrength}{\effmgFieldStrength}\approx 114.25.
\end{align}
%
Since the semiclassical parameter \(\scp \approx 1.19\cdot 10^{-8}\) is very small, the dynamics of the wave function are highly oscillatory in space and time, motivating the \GWP ansatz~\eqref{eq:gwp} for eliminating high oscillations. 
In \Cref{fig:penning-trap-projected-exact}, we illustrate the corresponding proton dynamics for the typical trap parameters of \Cref{tab:trap-frequencies}. The parameters of the initial \GWP $\sol(0,x)$ are specified in \eqref{eq:penning-IV}.  
The plot depicts the exact trajectory of the center in blue. One cycle of the cycloidal motion is highlighted in orange and one of the axial motion in red.

\section{Variational approximation}\label{sec:num-approx}

In the semiclassical regime, the solution of the Schr\"odinger equation \eqref{eq:sproblem} is highly oscillatory and well localized. We thus seek approximations within the manifold of complex Gaussian wave packets 
\begin{align} 
	\Mf = \Bigl\{ 
	&\vsol \in L^2(\bbR^\dim)  \bigm|
	\vsol(x) = \exp{\Bigl(\frac{\ii}{\scp}\bigl(\frac12(x-\pos)^\top \wm (x-\pos)+(x-\pos)^\top \mom +\pha\bigr)\Bigr) } , \nonumber
	\\
	&\pos,\,\mom\in \bbR^\dim, \,
	\wm = \wm[\top] \in \bbC^{\dim\times \dim}, \,
	\im \wm 
	\text{ positive definite}, \pha\in \bbC
	\Bigr\} . \label{eq:gwp}
\end{align}
We construct the optimal approximation $\vsol(t)\approx\sol(t)$, $\vsol(t)\in\Mf$, in the sense that the time-derivative $\pt\vsol(t)$ minimizes the residual,
\begin{equation}\label{eq:min}
    \normLtwo{\ii\scp \pt \vsol(t) - \Ham(t)\vsol(t)} = \min_{\pt\vsol(t)} !
\end{equation}
We consider initial data with
$\sol_0 = \vsol_0\in \Mf$ and $\normLtwo{\vsol_0} =1$,  
and mention in passing, that the variational approximation \eqref{eq:min} is norm preserving in any case and  energy preserving for time-independent Hamiltonians,
\begin{align}
\normLtwo{\vsol(t)} = \normLtwo{\vsol_0},\quad
\langle H\rangle_{\vsol(t)} = 
\langle H\rangle_{\vsol_0}\quad \text{for all}\ t,
\end{align}
 see \cite[\S II.1.5]{Lub08_book}. 

\subsection{Variational equations of motion}\label{subsec:eqmo}
In earlier work \cite{BurDHL24}, we  derived ordinary differential equations for the parameters $\pos_t$, $\mom_t$, $\wm_t$, $\pha_t$ of the approximating \GWP, which we reproduce here. We denote the classical Hamiltonian function for charged particles in a magnetic field by
\[
\cham(t,x,\xi) = \frac12\left(\xi - \mgPot(t,x)\right)^2 + \mltPot(t,x),\quad (t,x,\xi)\in\bbR\times\bbR^\dim\times\bbR^\dim.
\]
Then, the residual minimization in \eqref{eq:min} implies that
\begin{subequations} \label{eq:eom:var-all}
\begin{align+}\label{eom:var}
    \dot\pos &= \langle \partial_\xi \cham\rangle_{\vsol},\quad
    \dot\mom = -\langle \partial_x \cham\rangle_{\vsol},\quad
    \dot\wm = -\mathcal B(\wm),\\
    \dot\pha &= -\langle \cham\rangle_{\vsol} + 
    \frac{\scp}{4}\tr\!\left(\mathcal B(\wm)\ICinv\right) + p^\top\langle \partial_\xi \cham\rangle_{\vsol},
\end{align+}
where the complex matrix $\mathcal B(\wm)\in\bbC^{\dim\times\dim}$, that depends on a Gaussian average, is given by
\begin{equation}  \label{eq:BC-def}
    \mathcal B(\wm) = \begin{pmatrix}\Id & \wm\end{pmatrix} 
 \langle \nabla^2 \cham\rangle_{\vsol} \begin{pmatrix}\Id\\ \wm\end{pmatrix} 
 .
\end{equation}
\end{subequations}
The averages 
$\langle \partial^\alpha h\rangle_{\vsol} = \langle \op(\partial^\alpha h)\rangle_{\vsol}$ use the standard Weyl quantization of the derivatives of the Hamiltonian function.
In \Cref{sec:teom}, we reformulate these averages such that the equations of motion become amenable to a Boris-type time discretization.

\subsection{Asymptotic accuracy}\label{sec:accuracy}
The variational \GWP $\vsol(t)$ determined by \eqref{eq:min} is the exact Schr\"odinger solution $\sol(t)$, if the magnetic potential $\mgPot(t,\cdot)$ is 
linear and the electric potential $\mltPot(t,\cdot)$ quadratic with respect to position. In particular, the dynamics in a Penning trap are exactly described. More generally, if $\mgPot$ and $\mltPot$ are sub-linear and sub-quadratic in the following sense,
\begin{align}
    &\partial^\alpha_{x} \mgPot(t,\cdot), \quad \partial^{\beta}_{x} \mltPot(t,\cdot) 
    \qquad \text{bounded for all}\ |\alpha|\ge 1,\ |\beta|\ge 2,
\end{align}
then the following bounds for the norm and the observable error can be proven, see \cite[Theorems 3.8 \& 3.10]{BurDHL24}. Given a finite time horizon $[0,T]$,  the norm error and the observable 
error for expectation values with respect to a linear operator $\obs$ satisfy
\begin{align+}\label{eq:error_all}
&\normLtwo{\sol(t)-\vsol(t)} \le c\, t\,\sqrt\scp,\qquad
\left|\langle \obs\rangle_{\sol(t)} - \langle \obs\rangle_{\vsol(t)}\right| \le C\, t\, \scp^2
\end{align+}
for all $t\in[0,T]$, 
where the constants $c,C>0$ are independent from $\scp$ and $t$, but depend on a lower bound for the eigenvalues of $\IC$ on $[0,T]$. It is worth pointing out the high observable accuracy.

\subsection{Classical versus variational approximation}
Traditional  \GWP approximation evolves the centers 
according to the purely classical equations of motion
\[
\dot q = \partial_p h,\qquad \dot p = - \partial_q h,
\]
see \cite[\S4]{RobC21_book}. If the magnetic and the electric potential are linear respectively quadratic in position (as they are for a Penning trap), then the classical and the variational approximation coincide and yield the exact quantum solution. For general Hamiltonians, however, the approximations differ. The variational equations of motion \eqref{eq:eom:var-all} contain averages with respect to the approximating \GWP, which are computationally more costly than the pure point evaluations of the traditional approach. However, a traditional Gaussian approximation is neither energy preserving nor as accurate as a variational one, since the classical observable error is only first order in $\scp$, while the variational one is second order.

\subsection{Hagedorn parametrization} \label{subsec:hagedorn}
It is convenient to write the complex symmetric width matrix $\wm$ in Hagedorn's parametrisation as 
\begin{equation+} \label{eq:Hag_fac}
	\wm =\RC + \ii \IC = \fcP \fcQ^{-1} \quad \mathrm{and}\quad \IC = (\fcQ\fcQ^*)^{-1},
\end{equation+}
with two real symmetric matrices $\RC,\IC$ and two invertible complex  matrices $\fcQ,\fcP$ that satisfy the symplecticity condition $\fcQ^\top\fcP - \fcP^\top\fcQ = 0$, $\fcQ^*\fcP - \fcP^*\fcQ = 2\ii\mathrm{Id}$. Such a decomposition of complex symmetric matrices with positive definite imaginary part is unique up to unitary factors, see \cite[Chapter V]{Lub08_book}. The Riccati-type equation \eqref{eom:var} for the evolution of the width matrix $\wm$ takes in Hagedorn's parametrisation the form
\[
\dot \fcQ = \mean{\partial_{pq}h}Q + \mean{\partial_{pp}h}P,\quad
\dot \fcP = -\mean{\partial_{qq}h}Q - \mean{\partial_{qp}h}P.
\]
Unfortunately, we can express $\RC = \frac12 (\fcP \fcQ^{-1} + \fcQ^{-*} \fcP^*)$ only in terms of both $\fcQ$ and $\fcP$, which will have implications for the integrators to be developed.

\section{Transformation of equations of motion}
\label{sec:teom}

We transform the variational equations of motion  in such a way that they structurally mimic the classical equations 
for a charged particle. 

\subsection{Averaged magnetic momenta}
In a first step, we rewrite the variational equations of motion such that an averaged version of the usual magnetic momenta becomes visible. 

\begin{lemma} \label{lem:var_equations}
The variational equations of motion \eqref{eq:eom:var-all} for the parameters of a Gaussian wave packet are equivalent to
\begin{subequations} \label{eq:ham-all}
 \begin{align+}
    \dot{\pos} &=  \mom - \mean{\mgPot}, \quad 
    \dot{\mom} =  \mean{J_{\mgPot}^\top (\xi-\mgPot) - \nabla\mltPot}, 
    \label{eq:ham-qp} \\
    \dot\wm &= -\mean{\partial^2_{x} \cham} + \mean{\jacobian{\mgPot}^\top}\wm + \wm \mean{\jacobian{\mgPot}} - \wm^2,
    \label{eq:eqmo_C-l31}\\
     \dot\pha &= -\langle \cham\rangle_{\vsol} + 
    \frac{\scp}{4}\tr(\mathcal B(\wm)\ICinv) + \mom^\top(\mom-\mean{\mgPot}),
    \label{eq:eqmo_zeta-l31}
\end{align+}  
\end{subequations}
where $\mean{a} = \mean{\op(a)}$ for any smooth  $a:\bbR^{2\dim}\to\bbR$, $(x,\xi)\mapsto a(x,\xi)$. In Hagedorn's parametrisation, the matrix evolution \eqref{eq:eqmo_C-l31} satisfies
\begin{subequations}\label{eq:ham-Q-P}
\begin{align+}
\dot \fcQ 
&= 
\fcP -\mean{\jacobian{\mgPot}} \fcQ, \label{eq:ham-Q} \\
\dot \fcP 
&= 
\mean{\jacobian{\mgPot}^\top}\fcP-\mean[\big]{\jacobian{\mgPot}^\top\jacobian{\mgPot} - 
\sum_{k=1}^\dim \nabla^2 \mgPot_k(\xi_k-\mgPot_k) 
+ 
\nabla^2\mltPot}\fcQ. \label{eq:ham-P}
\end{align+}
\end{subequations}
\end{lemma}

\begin{proof}  Since the partial derivatives of the classical Hamiltonian function $\cham(x,\xi) = \frac12(\xi-\mgPot(x))^2 + \mltPot(x)$ satisfy
\begin{align}
\partial_\xi \cham(x,\xi) = \xi-\mgPot(x),\qquad
\partial_x \cham(x,\xi) = -\jacobian{\mgPot}(x)^\top(\xi-\mgPot(x)) + \nabla\mltPot(x),
\end{align}
we have \eqref{eq:ham-qp}. We next turn to the second order derivatives of~$\cham$. Denoting the partial derivatives of the magnetic and the electric potential shortly by $\partial_1,\ldots,\partial_\dim$, we have
\begin{align}
\partial_{x_m} \partial_{x_\ell} \cham &= \sum_{k=1}^\dim \left( \partial_m \mgPot_k \partial_\ell \mgPot_k - (\xi_k-\mgPot_k)\partial_m\partial_\ell \mgPot_k\right) + \partial_m\partial_\ell\mltPot,\\
\partial_{x_m} \partial_{\xi_\ell}\cham &= -\partial_m \mgPot_\ell,\qquad
\partial_{\xi_m} \partial_{x_\ell} \cham = -\partial_\ell \mgPot_m,
\end{align}
which gives a Hessian matrix 
\[
\nabla^2 \cham = 
\begin{pmatrix}\partial_{x}^2 \cham & \partial_{x\xi} \cham\\ \partial_{\xi x}\cham & \partial_{\xi}^2\cham\end{pmatrix} =
\begin{pmatrix} \partial_{x}^2\cham & -\jacobian{\mgPot}^\top\\ - \jacobian{\mgPot} & \mathrm{Id}\end{pmatrix}
\]
with
$\partial_{x}^2\cham = \jacobian{\mgPot}^\top\jacobian{\mgPot} - \sum_{k=1}^\dim \nabla^2 \mgPot_k(\xi_k-\mgPot_k) + \nabla^2\mltPot$.
By \eqref{eq:BC-def} we then infer that 
\begin{align}
\mathcal B(\wm) 
&= \mean{\partial^2_{x} \cham} + \mean{\partial_{x\xi}\cham}\wm + \wm \mean{\partial_{\xi x}\cham} + \wm\mean{\partial^2_\xi \cham}\wm\\
&= \mean{\partial^2_{x} \cham} - \mean{\jacobian{\mgPot}^\top}\wm - \wm \mean{\jacobian{\mgPot}} + \wm^2.
\end{align}
Hence, the variational equations of motion \eqref{eq:eom:var-all} are equivalent to  \eqref{eq:ham-all}.
\end{proof}

This formulation of the equations of motion prominently features averaged magnetic momenta 
\begin{equation}  \label{eq:mag-momenta}
\vel \da \mom - \mean{\mgPot}, \qquad 
\fcVel \da \fcP - \mean{ \jacobian{\mgPot}}\fcQ. 
\end{equation}
We aim at rewriting them in terms of the real vector $\vel(t)\in\bbR^\dim$ and the complex matrix $\fcVel(t)\in\bbC^{\dim\times\dim}$. 

\subsection{Equations for the center and the width}
Since the averaged magnetic momenta contain averages of the magnetic field and of its Jacobian, we need a compact way of assessing the time derivative of averages in the spirit of a magnetic Ehrenfest-type theorem.

\begin{prop}\label{prop:Ehrenfest}
For any smooth function \(\somePots\colon\Rd\to\bbR\), its average with respect to the variational \GWP \(\vsol = \vsol(t)\) satisfies
\begin{align}
\frac{d}{dt}\mean{\somePots} 
&= 
\mean{\nabla\somePots}^\top\vel +
\frac\scp2\tr\!\left(\mean{\nabla^2 \somePots} (\fcVel\fcQ^* - \ii\Id) \right),
\end{align}
where $\fcVel\fcQ^* - \ii\Id$ is a real matrix.
\end{prop}

\begin{proof}
We differentiate the position density with respect to time,
\begin{align}
    \pt\abs{\vsol(t, x)}^2 
    &= \pt \exp\!\left(-\frac1\scp (x-\pos)^\top \ImC (x-\pos) - \frac2\scp \ImPha\right)\\
    &= \abs{\vsol(t, x)}^2\bigg( -\frac1\scp (x-\pos)^\top \dot\ImC (x-\pos) +\frac2\scp (x-\pos)^\top\ImC \dot\pos - \frac2\scp \dot\ImPha\bigg)\\
    &=        
    \abs{\vsol(t, x)}^2\bigg(\frac{2}{\scp}(x-\pos)^\top
    \big(
        \ReC - \mean{\jacobian{\mgPot}^\top}
    \big)\ImC(x-\pos) \nonumber\\
    &\qquad\qquad\quad
    + \frac{2}{\scp}(x-\pos)^\top\ImC\vel - \tr(\ReC)\bigg),
\end{align}
where we have used that, by Lemma~\ref{lem:var_equations}, the symmetry of $\ReC$ and $\ImC$, and $\tr(\mean{\jacobian{\mgPot}})=0$ since $\nabla \cdot A=0$,
\[
    \dot\pos = \vel,\quad
    \dot\ImC = \ImC(\mean{\jacobian{\mgPot}}-\ReC) + (\mean{\jacobian{\mgPot}^\top}-\ReC)\ImC,\quad
    \dot\ImPha = \frac\scp2 \tr(\ReC).
\]
%
Therefore, we obtain
\begin{equation+} 
\begin{aligned}
    \frac{d}{dt}\mean{\somePots} &= \frac{2}{\scp}\mean[\Big]{\somePots (x-\pos)^\top
        \big(\ReC - \mean{\jacobian{\mgPot}^\top}\big)\ImC(x-\pos)}  
         \\
        &\quad+ \frac{2}{\scp}\mean[\big]{\somePots\ImC (x-\pos)}^\top \vel - \mean{\somePots}\tr(\ReC)\\
        &=  \mean{\somePots}  \tr(\ReC - \mean{\jacobian{\mgPot}^\top}) + \mean{\nabla \somePots}^\top \vel - \mean{\somePots}\tr(\ReC) \\
        &\quad+ \frac{\scp}2 \tr(\mean{\nabla^2 \somePots} \ImCinv (\ReC - \mean{\jacobian{\mgPot}^\top}))\\
        &= \mean{\nabla \somePots}^\top \vel + \frac{\scp}2 \tr(\mean{\nabla^2 \somePots} \ImCinv (\ReC - \mean{\jacobian{\mgPot}^\top})),
        \end{aligned} 
\end{equation+} 
where the second equation relies on~\Cref{lem:Gauss_calc} with matrix $\someMatrix = \big(\ReC - \mean{\jacobian{\mgPot}^\top}\big)\ImC$ and the last equation on $\nabla\cdot\mgPot = 0$. It remains 
to express the occurring matrices in terms of $\fcQ$ and $\fcVel$. Using the properties of Hagedorn's parametrization from Sec.~\ref{subsec:hagedorn} yields
\begin{align}
\ImCinv\ReC &= \frac12\fcQ\fcQ^*(\fcP\fcQ^{-1}+\fcQ^{-*}{\fcP}^*)
\nonumber
\\
&=\frac12 \left( \fcQ({\fcP}^*\fcQ + 2\ii\Id)\fcQ^{-1} + \fcQ{\fcP}^*\right)
\nonumber\\
&=\fcQ({\fcVel}^* + \fcQ^*\mean{\jacobian{\mgPot}^\top}) + \i\Id. 
\label{eq:CIinvCR}
\end{align}
Hence we have $\ImCinv(\ReC-\mean{\jacobian{\mgPot}^\top}) = \fcQ{\fcVel}^* + \ii\Id$, and use the trace identity 
$\tr(MN) = \tr(MN^*)$ for the real symmetric matrix $M=\mean{\nabla^2\somePots}$ and the real matrix 
$N = \fcQ{\fcVel}^* + \ii\Id$. 
\end{proof}

Equipped with an equation for the time evolution of averages, we can determine the variational equations of motion solely in terms of the averaged magnetic momenta.

\begin{theorem}\label{Lemma:euler-lagrange-q-v}
    We consider the averaged magnetic momenta \eqref{eq:mag-momenta} and denote the magnetic field $\mgField = \curl\mgPot$. 
    The variational equations of motion for the Gaussian wave packet's center and width, equations~\eqref{eq:ham-qp} and~\eqref{eq:ham-Q-P}, can be transformed as  
    \begin{subequations}\label{eq:euler-lagrange-q-v}
        \begin{align+}
        \label{eq:eul-lag-q-v}
            \dot{\pos} &= \vel,  &
            \dot{\vel} &= \vel\times\mean{\mgField} + \auxPot, \\
        \label{eq:eul-lag-Q-Y}
            \dot \fcQ &= \fcVel,  & 
            \dot{\fcVel} &= \fcVel\times\mean{\mgField} + \auxauxPot\fcQ,
        \end{align+}
    \end{subequations}
where the matrix 
\[
\fcVel\times\mean{\mgField} \da (\upsilon_1\times \mean{\mgField},\ldots,\upsilon_d\times\mean{\mgField})
\] 
is given by the cross product of the column vectors $\upsilon_1,\ldots,\upsilon_d$ of $\fcVel$ with $\mean{\mgField}$. The real vector field $\auxPot$ satisfies
    \begin{align}
    \auxPot 
    &= 
    -
    \mean{\nabla\mltPot} 
    -
    \mean{\partial_t\mgPot} 
    + 
    \mean{\jacobian{\mgPot}^\top}\mean{\mgPot} 
    - 
    \mean{\jacobian{\mgPot}^\top\mgPot} \\
    &+
    \frac\scp2\Bigl(
    \tr\!\left( \mean{\partial_k\jacobian{\mgPot}^\top-\nabla^2\mgPot_k} (\fcVel\fcQ^* - \ii\Id) + \mean{\partial_k\jacobian{\mgPot}^\top} \mean{\jacobian{\mgPot}}\fcQ\fcQ^*\right)    
    \Bigr)_{k=1}^\dim, 
    \end{align}
and the real matrix potential $\auxauxPot$ can be written as
\begin{align}
&\auxauxPot =  
    - \mean{\nabla^2\mltPot} 
    - \mean{ \jacobian{\partial_t\mgPot}}
    + \mean{\jacobian{\mgPot}^\top}\mean{\jacobian{\mgPot}} 
    -     \mean{\jacobian{\mgPot}^\top\jacobian{\mgPot}} \\
&+
\sum_{m=1}^\dim 
    \mean{\nabla^2 \mgPot_m}\mean{\mgPot_m} 
    - \mean{(\nabla^2 \mgPot_m) \mgPot_m}
+  
\left(
    \mean{\partial_k\partial_\ell\mgPot_m-\partial_m\partial_\ell\mgPot_k}\vel_m
\right)_{k,\ell=1}^\dim 
 \\
& +\frac{\scp}{2} \left( 
\tr\left( \mean{\partial_k\partial_\ell\jacobian{\mgPot}^\top-\nabla^2\partial_\ell \mgPot_k}  (\fcVel\fcQ^*-\ii\Id) + \mean{\partial_k\partial_\ell \jacobian{\mgPot}^{\top}} \mean{\jacobian{\mgPot}}\fcQ\fcQ^* \right) \right)_{k,\ell=1}^\dim.
\end{align}
      
\end{theorem}

\begin{proof}
By \eqref{eq:ham-qp}, the time derivative of the magnetic momenta satisfies
\begin{align} \label{eq:dot-vel}
\dot\vel &= \dot\mom - \frac{d}{dt} \mean{\mgPot}
= \mean{\jacobian{\mgPot}^\top(\xi-\mgPot)-\nabla\mltPot} - \frac{d}{dt} \mean{\mgPot}.
\end{align}
Hence, we have to work on $\mean{\jacobian{\mgPot}^\top\xi}$. By symbolic Weyl calculus, see for example \cite[\S2.4]{RobC21_book}, we expand the operator of the product as 
\begin{align}
    \op(\jacobian{\mgPot}^\top \xi) &= \jacobian{\mgPot}^\top\op(\xi) + \frac{\ii\scp}{2} \op(\{\jacobian{\mgPot}^\top,\xi\})
    = \jacobian{\mgPot}^\top (-\ii\scp\nabla),
\end{align}
where the last equation uses that $\nabla\cdot \mgPot = 0$ implies that
$
\{\jacobian{\mgPot}^\top,\xi\} = -\sum_{k=1}^\dim \partial_{x_k}(\jacobian{\mgPot}^\top)\partial_{\xi_k}\xi = -\sum_{k=1}^\dim \partial_{x_k} \nabla \mgPot_k = 0.
$
Therefore, 
\begin{align}
\mean{\jacobian{\mgPot}^\top\xi} &= \langle \vsol,\jacobian{\mgPot}^\top (-\ii\scp\nabla)\vsol\rangle 
    = \mean{\jacobian{\mgPot}^\top \wm(x-\pos)} + \mean{\jacobian{\mgPot}^\top}\mom
\end{align}
due to 
$-\ii\scp\nabla\vsol(x) = (\wm(x-\pos) + \mom)\vsol(x)$. Then, we apply \Cref{lem:Gauss_calc} with $\somePot = \jacobian{\mgPot}^\top\wm$,  
\begin{align}
    \mean{\jacobian{\mgPot}^\top \wm(x-\pos)} &= 
\frac\scp2 \sum_{\ell=1}^\dim 
\left(  \left(\mean{\partial_\ell\jacobian{\mgPot}^\top} \wm\ImCinv\right)_{k\ell} \right)_{k=1}^\dim\\
&=    
    \frac\scp2 \sum_{m,\ell=1}^\dim \left(\mean{ \partial_\ell\partial_k \mgPot_m} \left(\fcVel\fcQ^* + \mean{\jacobian{\mgPot}}\fcQ\fcQ^* - \ii\Id\right)_{m\ell} \right)_{k=1}^\dim\\
    &=
    \frac\scp2 \left( \tr\!(\mean{\partial_k\jacobian{\mgPot}}^\top ( \fcVel\fcQ^* -\ii\Id  + \mean{\jacobian{\mgPot}}\fcQ\fcQ^*) \right)_{k=1}^\dim,\label{eq:meanC}
\end{align}
since $\wm\ImCinv = (\fcP\fcQ^{-1})(\fcQ\fcQ^*) = (\fcVel + \mean{\jacobian{\mgPot}}\fcQ)\fcQ^*$ and $\Div A = 0$.
We thus obtain
\begin{align}
    \mean{\jacobian{\mgPot}^\top(\xi-\mgPot)} &= 
    \mean{\jacobian{\mgPot}^\top} \vel 
    + 
    \mean{\jacobian{\mgPot}^\top}\mean{\mgPot} 
    - 
    \mean{\jacobian{\mgPot}^\top\mgPot} \\
    &+ 
   \frac\scp2 \left( \tr\!(\mean{\partial_k\jacobian{\mgPot}}^\top (\fcVel\fcQ^* - \ii \Id + \mean{\jacobian{\mgPot}}\fcQ\fcQ^*) \right)_{k=1}^\dim.
\end{align}
Now, we use \Cref{prop:Ehrenfest} for each component of the magnetic potential,
\[
\frac{d}{dt} \mean{\mgPot} 
= 
\mean{\partial_t\mgPot} 
+ 
\mean{\jacobian{\mgPot}}\vel 
+
\frac\scp2\Bigl(
\tr\!\bigl(
   \mean{\nabla^2 \mgPot_k}(\fcVel\fcQ^*-\ii\Id)
\bigr)
\Bigr)_{k=1}^\dim.
\]
When collecting all the terms that originated in \eqref{eq:dot-vel}, we observe the occurrence of $\mean{\jacobian{\mgPot}^\top}\vel$ and $-\mean{\jacobian{\mgPot}}\vel$, which combine according to 
\begin{equation}\label{eq:jacobian-to-cross}
        \mean{\left(\jacobian{\mgPot}^\top - \jacobian{\mgPot}\right)}\vel 
        = \vel\times\mean{\mgField}.
    \end{equation}
We thus arrive at the claimed equation \eqref{eq:eul-lag-q-v}. As for the matrix $\fcVel = \fcP - \mean{\jacobian{\mgPot}}\fcQ$, 
we use the equations of motion \eqref{eq:ham-Q-P}, which contain the average of $\partial_x^2\cham= \jacobian{\mgPot}^\top\jacobian{\mgPot} - \sum_{m=1}^\dim \nabla^2 \mgPot_m(\xi_m-\mgPot_m) + \nabla^2\mltPot$. We have 
\[
\dot\fcP 
= 
\mean{\jacobian{\mgPot}^\top}\left(\fcVel + \mean{\jacobian{\mgPot}}\fcQ\right)
- \mean{\partial^2_x\cham}\fcQ.
\]
For computing $\mean{\partial^2_x\cham}$, we observe that 
\begin{align}
\sum_{m=1}^\dim \op(\nabla^2 \mgPot_m\xi_m) 
&= 
\sum_{m=1}^\dim \left( 
    \nabla^2 \mgPot_m \op(\xi_m)
    + 
    \frac{\ii\scp}{2} \op\{\nabla^2\mgPot_m,\xi_m\} 
\right) \\
&= 
\sum_{m=1}^\dim \nabla^2 \mgPot_m (-\ii\scp\partial_m),
\end{align}
where the last equation uses that $\sum_{m=1}^\dim\{\nabla^2\mgPot_m,\xi_m\} = 0$ due to $\nabla\cdot\mgPot = 0$. 
Arguing as for \eqref{eq:meanC}, we obtain
\begin{align}
&\sum_{m=1}^\dim \mean{\nabla^2 \mgPot_m(\xi_m-\mgPot_m)} \\
&= 
\sum_{m=1}^\dim \left(
    \mean{\nabla^2 \mgPot_m \bigl(\wm(x-\pos)\bigr)_m} 
    -
    \mean{(\nabla^2 \mgPot_m) \mgPot_m}
    + 
    \mean{\nabla^2 \mgPot_m}\mom_m
\right)\\
&= 
\sum_{m=1}^\dim \left(
    \mean{\nabla^2 \mgPot_m}\left(\vel_m + \mean{\mgPot_m}\right) 
    - \mean{(\nabla^2 \mgPot_m) \mgPot_m}
\right) \\
&\quad + 
\frac{\scp}{2}\sum_{m,n=1}^\dim \mean{\nabla^2 \partial_n\mgPot_m} 
\left(\fcVel\fcQ^* -\ii\Id + \mean{\jacobian{\mgPot}}\fcQ\fcQ^*) \right)_{mn},
\end{align}
and therefore
\begin{align}
\mean{\partial_x^2\cham} &= \mean{\jacobian{\mgPot}^\top\jacobian{\mgPot}+\nabla^2\mltPot} 
- \sum_{m=1}^\dim \left(
    \mean{\nabla^2 \mgPot_m}\left(\vel_m + \mean{\mgPot_m}\right) 
    - \mean{(\nabla^2 \mgPot_m) \mgPot_m}\right)\\
    &-\frac{\scp}{2}  \left( \tr\!\left(\mean{\partial_k\partial_\ell \jacobian{\mgPot}^\top} 
\left( \fcVel\fcQ^* - \ii\Id + \mean{\jacobian{\mgPot}}\fcQ\fcQ^*\right)\right)\right)_{k,\ell=1}^\dim .\label{eq:2xdiffham}
\end{align}
We next use \Cref{prop:Ehrenfest} applied to each component of the Jacobian matrix $\jacobian{\mgPot}= (\partial_\ell \mgPot_k)_{k,\ell=1}^\dim$ and obtain
\begin{align}
    &\frac{d}{dt}\mean{ \jacobian{\mgPot}} 
    = 
    \mean{ \jacobian{\partial_t\mgPot}}\\ 
    &+ 
    \sum_{m=1}^\dim \left(
        \mean{\partial_m\partial_\ell\mgPot_k}\vel_m
    \right)_{k,\ell=1}^\dim 
    +
    \frac\scp2\left(
        \tr\! \Bigl(
        \mean{\nabla^2 \partial_\ell\mgPot_k}(\fcVel\fcQ^*-\ii\Id)
        \Bigr)
    \right)_{k,\ell=1}^\dim .
\end{align}
Now we combine and arrive at
\begin{align}
\dot\fcVel 
&= 
\dot\fcP 
- 
\left(\frac{d}{dt}\mean{\jacobian{\mgPot}}\right) \fcQ 
- 
\mean{ \jacobian{\mgPot}} \dot\fcQ 
=  
\mean{\jacobian{\mgPot}^\top - \jacobian{\mgPot}} \fcVel 
+ \auxauxPot\fcQ
\end{align}
with the matrix potential $\auxauxPot$ of the claimed form. 
\end{proof}

\begin{remark}[Linear potential \(\mgPot\)]\label{rem:linear}
    For a linear magnetic potential \(\mgPot\), all higher order derivatives vanish and the average \(\mean{\mgPot} = \mgPot(\pos)\) is a point evaluation. Thus, the equations of motion of \Cref{Lemma:euler-lagrange-q-v} simplify to
    \begin{subequations}\label{eq:simple-euler-lagrange-system}
        \begin{align+}
            \dot{\pos} &= \vel, &
            \dot{\vel} &= \vel\times\mgField(\pos) - (\pt\mgPot(\pos) + \mean{\nabla\mltPot}), \label{eq:simple-qv}\\
            \dot{\fcQ} &= \fcVel, &
            \dot{\fcVel} &= \fcVel\times\mgField(\pos) - (\pt\jacobian{\mgPot}(\pos) + \mean{\nabla^2\mltPot})\fcQ.
            \label{eq:simple-QY}
        \end{align+}
    \end{subequations}
        The Penning trap, see \Cref{subsec:intro-Penning-trap} and \Cref{subsec:Penning-trap}, with its quadratic electric potential, has even simpler equations of motion, since also the averages of the electric potential collapse to $\mean{\nabla\mltPot} = \nabla\mltPot(\pos)$ and $\mean{\nabla^2\mltPot} = \nabla^2\mltPot(\pos)$.
\end{remark}

\subsection{Equations for the phase}\label{subsec:splitting_zeta}
The imaginary part of the complex parameter 
\(\pha = \RePha+\ii\ImPha\) carries the normalization of the Gaussian wave packet. We derive an explicit representation, which only depends on the determinant of the complex matrix \(\fcQ\). This representation will support our time discretization of \(\ImPha\) in \Cref{subsec:proposed-alg}.

\begin{lemma}
\begin{subequations} \label{eq:Pha-eqs-all}
    For the real part $\RePha = \re(\pha)$ we have
    \begin{equation}\label{eq:zeta_R}
   \dot\RePha = \frac12 {\vel}^2 + \mean{\mgPot}^\top\vel - \mean{\mltPot} + \frac\scp4 \tr(\mean{\partial_x^2\cham}\fcQ\fcQ^* - 2(\fcQ\fcQ^*)^{-1}),
   \end{equation}
    where the average $\mean{\partial^2_x\cham}$ is given in \eqref{eq:2xdiffham} with respect to $(\pos,\vel,\fcQ,\fcVel)$.    
   The imaginary part $\ImPha = \im(\pha)$ satisfies the normalization formula
    \begin{equation}\label{eq:relation_zeta_I}
        \ImPha(t) = \ImPha(0) + \frac{\scp}{2}\left(\ln\abs{\det Q(t)} - \ln\abs{\det Q(0)}\right).
    \end{equation}
\end{subequations}   
\end{lemma}

\begin{proof}
    We start with the normalization formula. 
    Since $\im\mathcal B(\wm) = -\mean{\jacobian{\mgPot}^\top}\ImC-\ImC\mean{\jacobian{\mgPot}} + \ReC\ImC + \ImC\ReC$ and 
    \(\Div \mgPot = 0\), we have 
    \begin{align}
    \im\left(\tr(\mathcal B(\wm)\ImCinv)\right) &= 2\,\tr(\ReC)
    = 2\,\tr(\re(\fcP\fcQ^{-1})) \\
    &= 2\,\tr(\re((\dot\fcQ+\mean{\jacobian{\mgPot}}\fcQ)\fcQ^{-1})) \\
    &= 2\,\tr(\re(\dot\fcQ\fcQ^{-1})).
    \end{align}
    We thus obtain from \eqref{eq:eqmo_zeta-l31} with Jacobi's formula
    \begin{align}
        \dot{\ImPha} &= \frac{\scp}{2}\tr(\re(\dot{\fcQ}\fcQ^{-1})) \nonumber \\
                      &= \frac{\scp}{4}\frac{1}{\abs{\det \fcQ}^2}\left(2\re\left(\overline{\det \fcQ}\det \fcQ~\tr\big(\dot{\fcQ}\fcQ^{-1}\big)\right)\right) \nonumber \\
                      &= \frac{\scp}{4}\frac{1}{\abs{\det \fcQ}^2}\left(2\re(\overline{\det \fcQ}~\pt(\det \fcQ))\right) \nonumber \\
                      &= \frac{\scp}{4}\pt\left(\ln\abs{\det \fcQ}^2\right) \ . \label{eq:eqmo-Imzeta}
    \end{align}
    Integrating~\eqref{eq:eqmo-Imzeta} from \(0\) to \(t\) leads 
to~\eqref{eq:relation_zeta_I}. For the real part, we have $\re\mathcal{B}(\wm) = \mean{\partial_x^2\cham} - \mean{\jacobian{\mgPot}^\top}\ReC - \ReC\mean{\jacobian{\mgPot}} + \ReC^2 - \ImC^2$ and thus
\[
\tr(\re\mathcal{B}(\wm)\ImCinv) = \tr\left(\left(\mean{\partial_x^2\cham}-2\mean{\jacobian{\mgPot}^\top}\ReC+ (\ReC^2-\ImC^2)\right)\ImCinv\right).
\]
Combining this with $-\mean{\cham}$, we use \Cref{Lemma:energy_equation} and obtain 
\begin{align}
&-\mean{\cham} + \frac\scp4 \tr(\re\mathcal{B}(\wm)\ImCinv) \\
&= 
-\frac12\mean{(p-\mgPot)^2} - \mean{\mltPot} + \frac\scp4 \tr(\mean{\partial_x^2\cham}\ImCinv - 2\ImC).
\end{align}
Next we observe that
\begin{align}
-\frac12\mean{(p-\mgPot)^2}  + \mom^\top\mean{\mom-\mgPot} &= \frac12 \mom^\top\mean{\mom-\mgPot} 
+ \frac12 \mean{\mgPot}^\top\mean{p-\mgPot} \\
&= \frac12 {\vel}^2 + \mean{\mgPot}^\top\vel.
\end{align}
Using that $\ImC = (\fcQ\fcQ^*)^{-1}$, the real part of the evolution equation \eqref{eq:eqmo_zeta-l31} can thus be written in the claimed form. 
\end{proof}

%

\section{Time integration for the equations of motion}\label{sec:time-int}

In this section we first briefly review the classical Boris algorithm. Afterwards, we present the new Boris-type algorithm and a modification of the classical Runge--Kutta method to solve \eqref{eq:euler-lagrange-q-v}  and~\eqref{eq:Pha-eqs-all}.
%

\subsection{Boris algorithm for classical equations of motion}\label{subsec:boris-algo}

The Boris algorithm was originally proposed in~\cite{Boris70} for solving the classical equations of motion 
\begin{equation}\label{eq:classical-EM-system}
        \dot{\pos} = \vel, \qquad
        \dot{\vel} = \vel\times \mgField + \elField.
\end{equation}
for charged particles in an \electroMagnetic field.
We consider a time-grid $\tn[n]$, $n\ge 0$, with step size 
\(\tau>0\).
Given approximations \(\pos[n]\approx\pos(\tn[n])\) and \(\vel[n-\frac12]\approx\vel(\tn[n-\frac12])\), the algorithm can be written as
\begin{subequations}\label{eq:Boris-Algo}
    \begin{align+}
        \vel[-] &= \vel[n-\frac12] + \frac{\tau}{2} \elField[n], 
        & \elField^n &= \elField(\tn[n], \pos[n]),\\
        \vel[+] -\vel[-] &= 
        \frac{\tau}{2} \bigl( \vel[+] +\vel[-] \bigr) \times \mgField[n],
        & \mgField[n] &= \mgField(\tn[n], \pos[n]),
        \label{eq:boris-impl}
        \\
        \vel[n+\frac12] &= \vel[+] + \frac{\tau}{2} \elField[n], 
        &&\\
        \pos[n+1] &= \pos[n] + \tau\vel[n+\frac12].
        &&
    \end{align+}
Note that the algorithm provides approximations on a staggered grid, where the velocities are only given at half time-steps. Approximations at $\tn[n]$ can be obtained by averaging
\begin{align+}\label{eq:average-vel}
    \vel[n] = \frac12\bigl(\vel[n+\frac12] + \vel[n-\frac12]\bigl) .
\end{align+}
Moreover, the scheme is explicit, since one can replace \eqref{eq:boris-impl} by
\begin{align+} \label{eq:boris-explicit}
            \vel[+] = \vel[-] + \left(\vel[-] + \vel[-]\times \frac{\tau}{2}\mgField^n \right)\times 
            \frac{\tau \mgField^n}{1+ \abs{\frac{\tau}{2} \mgField^n}^2},
\end{align+}
\end{subequations}
see, e.g., \cite{Bir18}.
The Boris algorithm is a \secondorder method which is not symplectic but conserves the phase-space volume as shown in~\cite{QinZXLST13}. A recent analysis was presented in~\cite{HaiL20}. 

\subsection{Boris-type algorithm}\label{subsec:proposed-alg}

We aim at solving the
\EulerLagrange system~\eqref{eq:euler-lagrange-q-v} together with the two phase equations
\eqref{eq:zeta_R} and~\eqref{eq:relation_zeta_I}. The former are closely related to the classical equations of motion of a charged particle \eqref{eq:classical-EM-system} except that $\mgField$ is replaced by an averaged field $\mean{\mgField}$ and that the fields $\elField$ and $\auxauxPot\fcQ$ do not only depend on \(\pos\) and \(\fcQ\) but also in a nontrivial way on \(\vel\) and \(\fcVel\).
While the means $\mean{\mgField}$ can be approximated by a suitable Gau{\ss}-Hermite quadrature formula, evaluating the fields $\elField$ and $\auxauxPot$ is more involved since they require approximations to \(\vel\) and \(\fcVel\) at time $\tn[n]$. Unfortunately, these quantities are only defined on the staggered grid $\tn[n\pm\frac12]$ and even worse, the update of $\fcVel$ is coupled to the evaluation of $\auxauxPot\fcQ$, rendering the scheme implicit.

To be more precise, it is the matrix 
$\fcVel$ which is necessary to compute the fields \(\elField\) and \(\auxauxPot\) in \Cref{Lemma:euler-lagrange-q-v}.
Averaging \(\fcVel\) as in \eqref{eq:average-vel} would lead to a nonlinear system in $\fcVel[n+\frac12]$. Therefore we propose the second-order extrapolation 
\begin{align+}\label{eq:extrap-upsilon}
    \fcVelExtr[n] = \frac32\fcVel[n-\frac12] - \frac12\fcVel[n-\frac32],
\end{align+}
and compute 
also $\auxPot[n]$ and $\auxauxPot[n]$ from it.
In our numerical experiments we saw that using fixed-point iterations to improve the accuracy of $\fcVelExtr[n]$ did not change the errors.
%

\subsection{Discretization of the phase}
Motivated by~\eqref{eq:relation_zeta_I}, we define the approximation to \(\ImPha(\tn[n+1])\) as
\begin{align+}\label{eq:relation_zeta_I_approx}
    \ImPha[n+1] = \ImPha[n] + \frac{\scp}{2}\left(\ln\abs{\det\fcQ[n+1]} - \ln\abs{\det\fcQ[n]}\right).
\end{align+}
This update formula ensures norm preservation of the wave packet.
\begin{lemma}
    Let \(\vsol[n]\) and \(\vsol[n+1]\) denote two \GWPs with parameters \((\pos[n], \vel[n], \fcQ[n], \fcVel[n], \pha[n])\) and \((\pos[n+1], \vel[n+1], \fcQ[n+1], \fcVel[n+1], \pha[n+1])\), respectively. Then~\eqref{eq:relation_zeta_I_approx} is equivalent to
    \begin{equation}
        \normLtwo{\vsol[n]} = \normLtwo{\vsol[n+1]} .
    \end{equation}
\end{lemma}
\begin{proof}
    Because of
    \begin{align}\label{eq:L2-norm-GWP}
        \normLtwo{\vsol[n]}^2
        = 
        \exp\bigl(-\tfrac{2}{\varepsilon}\zeta_I^n\bigr)(\varepsilon\pi)^\frac{d}{2}\abs{\det Q^n},
    \end{align}
    the norm preservation 
    is equivalent to
    \begin{equation}
        \exp\bigl(\tfrac{2}{\varepsilon}(\zeta_I^n-\zeta_I^{n+1})\bigr) 
        = 
        \frac{\abs{\det \fcQ^n}}{\abs{\det \fcQ^{n+1}}} .
    \end{equation}
    This proves the statement.
\end{proof}

Furthermore, for the integration of \eqref{eq:zeta_R}, we propose to use the midpoint rule with \timestep size \(2\tau\). This results in a two-step method. Since  \(\vel[n\pm\frac12]\) and \(\fcVel[n\pm\frac12]\) are already available, we apply averaging \eqref{eq:average-vel} instead of extrapolation to obtain approximations at $\tn[n]$. This results in
\begin{align} \label{eq:zeta-R-approx}
    \RePha[n+1] &= \RePha[n-1] + 2\tau \left( \frac12 (\vel[n])^2 + (\mean{\mgPot}^n)^\top\vel[n] -\mean{\mltPot}^n\right)\\  
    &\quad + \frac{\scp\tau}{2}\, \tr\!\left( \mean{\partial_x^2\cham}^n \fcQ[n](\fcQ[n])^* - 2(\fcQ[n]\fcQ[n]^*)^{-1} \right)  
 \end{align}

%

\subsection{Complete Boris-type algorithm}
Overall, combining the time discretization of center, width and phase of the Gaussian variational wave packet, we get the following algorithm for solving the equations of motion~\eqref{eq:euler-lagrange-q-v} and~\eqref{eq:Pha-eqs-all}. 
\begin{algo}[One \timestep with the Boris algorithm]\label{Algo:Boris}
    Input: Last steps \((\pos[n], \vel[n-\frac12], \fcQ[n], \fcVel[n-\frac12], \pha[n], \pha[n-1])\). \\[1ex]
    Output: Approximations \((\pos[n+1], \vel[n+\frac12], \fcQ[n+1], \fcVel[n+\frac12], \pha[n+1])\).
    \begin{itemize}
        \item compute \(\pos[n+1], \vel[n+\frac12]\) with the Boris algorithm~$\eqref{eq:Boris-Algo}$ applied to \eqref{eq:eul-lag-q-v}
        \item compute \(\fcQ[n+1], \fcVel[n+\frac12]\) with the Boris algorithm~$\eqref{eq:Boris-Algo}$ applied column-wise to ~\eqref{eq:eul-lag-Q-Y},
        
        \item compute \(\ImPha[n+1]\) and \(\RePha[n+1]\) from \eqref{eq:relation_zeta_I_approx} and \eqref{eq:zeta-R-approx}, respectively.
    \end{itemize}
\end{algo}

In general, one has to apply Gauss-Hermite quadrature rule to approximate the averages, see~\cite[Section~8]{LasL20} for details. Note that for strong magnetic fields, filtered variants of the Boris algorithms might be more efficient, cf.\ \cite{HaiLW20}.

\subsection{Modified classical Runge--Kutta method}\label{subsec:modifiedRK4}

As an alternative to the Boris-type algorithm, we propose a modification of the classical Runge--Kutta method (RK4) of order~4. The modification consists of updating the component \(\ImPha\) in each intermediate step by using \eqref{eq:relation_zeta_I_approx} to get the approximations to $\ImPha(t_{n+1/2})$ and $\ImPha(t_{n+1})$. All other components are updated by the standard RK4 procedure. 
In contrast to the original RK4 scheme, the modification automatically conserves the \(\Ltwo\)-norm of a \GWP.

\section{Numerical experiments}\label{sec:experiments}
In the following section, we present some numerical examples.

\subsection{Sublinear magnetic potential in two dimensions} \label{subsec:convg-plots}

If we consider the potentials
\begin{equation}\label{eq:sublinear-pot-expls}
    \mgPot(x, t) = \begin{pmatrix}
        \sin(x_1+x_2 + \alpha t) \\ -\sin(x_1+x_2 + \alpha t)
    \end{pmatrix}, \qquad
    \mltPot(x, t) = \sin(x_1+x_2)
\end{equation}
with $\alpha\in\{0,1\}$, we can calculate the occurring averages \(\mean{\mgPot}, \mean{\mltPot}\) analytically as
\begin{align}
    &\intRd \sin(x_1+x_2 + \alpha t)\abs{\vsol(x)}^2\dx \\
    &\qquad
    = 
    (\pi\scp)^{d/2}\mathrm{det}(L^{-1})
    \mathrm{exp}(-\tfrac{2}{\scp}\ImPha - \tfrac{\scp}{4}\mathds{1}^\top\fcQ\fcQ^*\mathds{1})
    \sin(\pos_1+\pos_2 + \alpha t)
\end{align}
where $\mathds{1}=(1 \ 1)^\top$ and \( (\fcQ\fcQ^*)^{-1} = LL^\top\) is the Cholesky decomposition. Since we compare the new time-integrators with the standard RK4 method, which is not norm-conserving, we do not assume normalization of $\vsol$. We use for the curl of a 2d vector potential \(\mgPot\) the convention \(\nabla\times\mgPot = \partial_1\mgPot_2 - \partial_2\mgPot_1\).
Initial values are chosen as
\begin{equation}\label{eq:sublinear-IV}
    \pos^0 
    = 
    \begin{pmatrix}
        0 \\ 0
    \end{pmatrix},
    \quad
    \mom^0
    = 
    \begin{pmatrix}
        1 \\ 0
    \end{pmatrix},
    \quad
    \fcQ^0
    =
    \begin{pmatrix}
        1 & 0 \\
        0 & 1
    \end{pmatrix},
    \quad
    \fcP^0
    =
    \begin{pmatrix}
        \mathrm{i} & 0 \\
        0 & \mathrm{i}
    \end{pmatrix},
    \quad
    \RePha[0] = 0. 
\end{equation}
The imaginary part of the phase \(\ImPha[0]\) is chosen such that the corresponding initial \GWP \(\vsol^0\) is normalized. Note that the normalization is \(\scp\) dependent and thus the initial value \(\ImPha[0]\) changes for different values of \(\scp\).

In \Cref{fig:sublinear-component-convg} we depict the component errors of a \GWP where we solve~\eqref{eq:euler-lagrange-q-v} with the Boris-type algorithm~\ref{Algo:Boris} (left) and the modified fourth order Runge--Kutta (mRK4) method (right). As end time we choose \(T=8\) and set \(\alpha = 1, \scp = 10^{-3}\). We compare
the numerical solution to a reference solution calculated by the standard RK4 method applied to~\eqref{eq:ham-all} and~\eqref{eq:ham-Q-P} with time step-size \(\tau = 10^{-4}\). As we see, the error decreases by order two for the Boris-type method and by order four for the mRK4 method with decreasing time step-size.
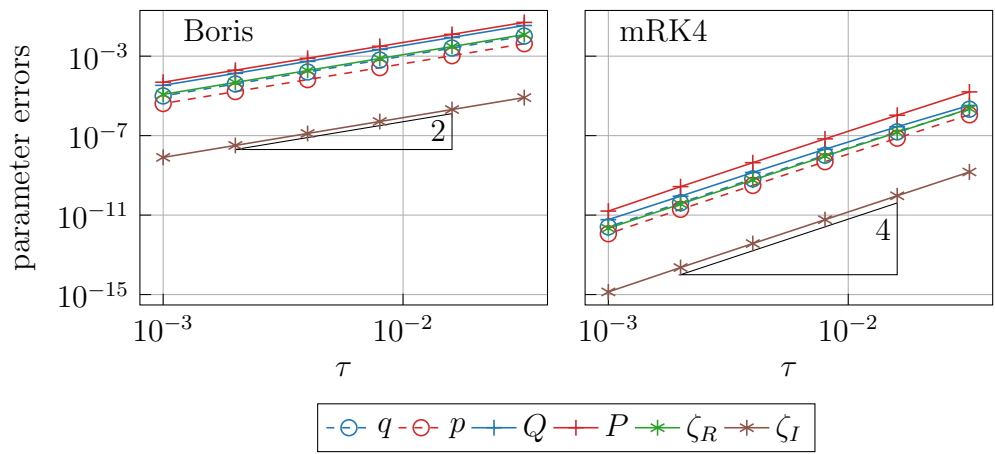
\begin{figure}[!htb]
    \centering
    \begin{tikzpicture}

    \definecolor{crimson2143940}{RGB}{214,39,40}
    \definecolor{darkgray176}{RGB}{176,176,176}
    \definecolor{darkorange25512714}{RGB}{255,127,14}
    \definecolor{forestgreen4416044}{RGB}{44,160,44}
    \definecolor{gray127}{RGB}{127,127,127}
    \definecolor{lightgray204}{RGB}{204,204,204}
    \definecolor{orchid227119194}{RGB}{227,119,194}
    \definecolor{sienna1408675}{RGB}{140,86,75}
    \definecolor{steelblue31119180}{RGB}{31,119,180}

    \begin{groupplot}[
        group style={group size=2 by 1, horizontal sep=0.5cm},
        width=7cm,
        height=5.5cm,
        enlarge x limits=false,
        legend columns=-1]
        
    \nextgroupplot[
        log basis x={10},
        log basis y={10},
        tick align=inside,
        tick pos=left,
        x grid style={darkgray176},
        y grid style={darkgray176},
        xmajorgrids,
        ymajorgrids,
        xminorticks=false,
        xlabel={$\tau$},
        ylabel={parameter errors},
        xmin=8e-4, xmax=0.04,
        ymin=3e-16, ymax=0.2,
        xmode=log, ymode=log,
        xtick style={color=black},
        ytick style={color=black},
        legend to name=sublinear-comp-group
    ]
    
    \addplot [semithick, steelblue31119180, dashed, mark=o, mark size=3, mark options={solid}]
    table {%
    0.032 0.0102388171065094
    0.016 0.00255358806458004
    0.008 0.000637971846711924
    0.004 0.000159460806293044
    0.002 3.98624890100995e-05
    0.001 9.96536470277718e-06
    };
    \addlegendentry{$q$}
    \addplot [semithick, crimson2143940, dashed, mark=o, mark size=3, mark options={solid}]
    table {%
    0.032 0.00422908675671902
    0.016 0.00105568579889727
    0.008 0.000263801280462916
    0.004 6.59396607836275e-05
    0.002 1.64840896169256e-05
    0.001 4.12092912785038e-06
    };
    \addlegendentry{$p$}
    \addplot [semithick, steelblue31119180, mark=+, mark size=3, mark options={solid}]
    table {%
    0.032 0.0350103971583852
    0.016 0.00871570180822293
    0.008 0.00217642674077701
    0.004 0.000543923644007725
    0.002 0.000135966061652004
    0.001 3.39901600916745e-05
    };
    \addlegendentry{$Q$}
    \addplot [semithick, crimson2143940, mark=+, mark size=3, mark options={solid}]
    table {%
    0.032 0.0500425186755286
    0.016 0.0125287560525743
    0.008 0.00313356927479516
    0.004 0.000783517205810596
    0.002 0.000195892043983455
    0.001 4.89744257444221e-05
    };
    \addlegendentry{$P$}
    \addplot [semithick, forestgreen4416044, mark=asterisk, mark size=3, mark options={solid}]
    table {%
    0.032 0.0121646373545197
    0.016 0.00303930260790786
    0.008 0.000759593893580934
    0.004 0.000189876384294685
    0.002 4.74669125036797e-05
    0.001 1.18664638799615e-05
    };
    \addlegendentry{$\zeta_R$}
    \addplot [semithick, sienna1408675, mark=asterisk, mark size=3, mark options={solid}]
    table {%
    0.032 8.22632712444695e-06
    0.016 2.0558362630557e-06
    0.008 5.13889773888845e-07
    0.004 1.28484922385513e-07
    0.002 3.21202519508283e-08
    0.001 8.03014573325794e-09
    };
    \addlegendentry{$\zeta_I$}

    \node at (1.7e-3,0.02) {Boris};
    
    \draw (0.016, 1.28e-6) -- (0.002, 2e-8) -- (0.016, 2e-8) -- cycle;
    \draw (0.014, 2e-6) coordinate[label=below:$2$];
    
    \nextgroupplot[
        log basis x={10},
        log basis y={10},
        scaled y ticks=manual:{}{\pgfmathparse{#1}},
        tick align=inside,
        tick pos=left,
        xmajorgrids,
        ymajorgrids,
        xminorticks=false,
        x grid style={darkgray176},
        y grid style={darkgray176},
        xlabel={$\tau$},
        xmin=8e-4, xmax=0.04,
        ymin=3e-16, ymax=0.2,
        xmode=log,
        ymode=log,
        xtick style={color=black},
        ytick style={color=black},
        yticklabel=\empty
    ]

    \addplot [semithick, steelblue31119180, dashed, mark=o, mark size=3, mark options={solid}]
    table {%
    0.032 2.14060514028276e-06
    0.016 1.54733236924675e-07
    0.008 1.03238865571962e-08
    0.004 6.65226469114236e-10
    0.002 4.1844724235266e-11
    0.001 2.54504572426307e-12
    };
    \addplot [semithick, crimson2143940, dashed, mark=o, mark size=3, mark options={solid}]
    table {%
    0.032 1.11361473778674e-06
    0.016 7.45553570672618e-08
    0.008 4.8795797589562e-09
    0.004 3.1160384655123e-10
    0.002 1.94949552205813e-11
    0.001 1.15144857006243e-12
    };
    \addplot [semithick, steelblue31119180, mark=+, mark size=3, mark options={solid}]
    table {%
    0.032 3.30742503902745e-06
    0.016 2.79285065622984e-07
    0.008 2.06118309711696e-08
    0.004 1.39040586693949e-09
    0.002 8.95885192072765e-11
    0.001 5.89488639574085e-12
    };
    \addplot [semithick, crimson2143940, mark=+, mark size=3, mark options={solid}]
    table {%
    0.032 1.59020046647456e-05
    0.016 1.06300020921442e-06
    0.008 6.85715253355482e-08
    0.004 4.34904961977116e-09
    0.002 2.71240690892684e-10
    0.001 1.59391222500026e-11
    };
    \addplot [semithick, forestgreen4416044, mark=asterisk, mark size=3, mark options={solid}]
    table {%
    0.032 2.26448379159194e-06
    0.016 1.45399416950909e-07
    0.008 9.20920850688844e-09
    0.004 5.79193581984327e-10
    0.002 3.61186636155253e-11
    0.001 2.159161738291e-12
    };
    \addplot [semithick, sienna1408675, mark=asterisk, mark size=3, mark options={solid}]
    table {%
    0.032 1.48319486900639e-09
    0.016 9.38142341588843e-11
    0.008 5.89148998683675e-12
    0.004 3.68945759360306e-13
    0.002 2.29399832463173e-14
    0.001 1.35872216255883e-15
    };

    \node at (1.7e-3,0.02) {mRK4};

    \draw (0.016, 4.096e-11) -- (0.002, 1e-14) -- (0.016, 1e-14) -- cycle;
    \draw (0.014, 2e-11) coordinate[label=below:$4$];
    
\end{groupplot}

\node (l1) at ($(group c1r1.south)!0.5!(group c2r1.south) + (0,-.25)$)
      [below, yshift=-2\pgfkeysvalueof{/pgfplots/every axis title shift}]
      {\ref*{sublinear-comp-group}}; 
    
\end{tikzpicture}
    \vspace{-4mm}
    \caption{
        Simulation of the motion of a particle in a magnetic field in dimension two with potentials \eqref{eq:sublinear-pot-expls} and initial values \eqref{eq:sublinear-IV}. 
        Errors of the numerical solution to~\eqref{eq:euler-lagrange-q-v} approximated by the 
        Boris-type algorithm (left) and the mRK4 method (right) measured in the Frobenius norm scaled with the inverse number of entries.
    }
    \label{fig:sublinear-component-convg}
\end{figure}

In \Cref{fig:sublinear-L2-convg-old} we illustrate the \(L^2\)-error of a \GWP with parameters calculated by the Boris-type method (left) and a \GWP with parameters calculated by the mRK4 method (right). Again we compute the error against a reference solution calculated with the standard RK4 method applied to~\eqref{eq:ham-all} and~\eqref{eq:ham-Q-P} with time stepsize \(\tau = 10^{-4}\). 
The \(L^2\)-norm between the two \GWPs is computed with a Gauss--Hermite quadrature rule. 
We see a reduction of order two for the Boris-type method and of order four for the mRK4 method in the \(L^2\)-norm. Moreover, we see that the error constant scales with as \(\scp^{-1}\) in both cases, which is supported by the theoretical result~\cite[Thm. 7.7]{LasL20}. Further note that the \(L^2\)-error between two normed \GWP is bounded by $2$, which explains the upper plateau in the left plot. The lower plateau in the right plot corresponds to the numerical computation of the underlying integral at almost machine precision.
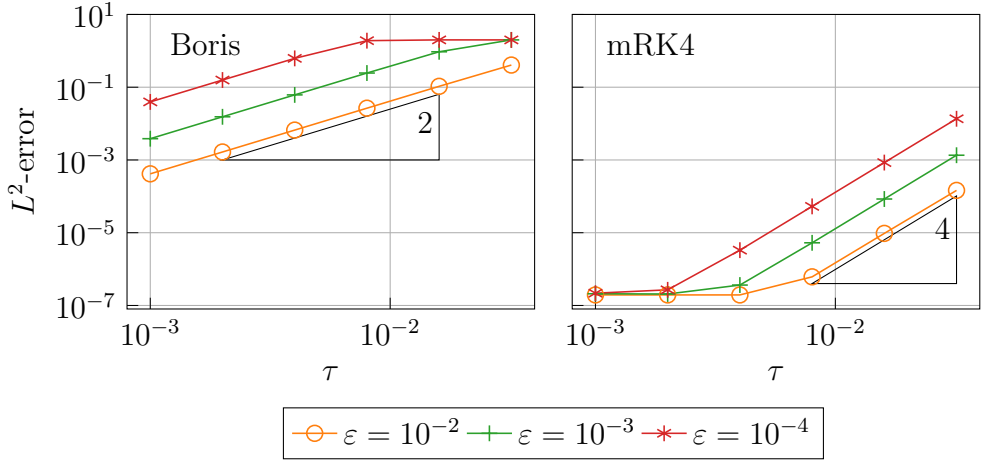
\begin{figure}[!htb]
    \centering
    \begin{tikzpicture}

    \definecolor{crimson2143940}{RGB}{214,39,40}
    \definecolor{darkgray176}{RGB}{176,176,176}
    \definecolor{darkorange25512714}{RGB}{255,127,14}
    \definecolor{forestgreen4416044}{RGB}{44,160,44}
    \definecolor{gray127}{RGB}{127,127,127}
    \definecolor{lightgray204}{RGB}{204,204,204}
    \definecolor{orchid227119194}{RGB}{227,119,194}
    \definecolor{sienna1408675}{RGB}{140,86,75}
    \definecolor{steelblue31119180}{RGB}{31,119,180}

    \begin{groupplot}[
        group style={group size=2 by 1, horizontal sep=0.5cm},
        width=7cm,
        height=5.5cm,
        enlarge x limits=false,
        legend columns=-1
        ]
        
    \nextgroupplot[
        log basis x={10},
        log basis y={10},
        tick align=inside,
        tick pos=left,
        x grid style={darkgray176},
        y grid style={darkgray176},
        xmajorgrids,
        ymajorgrids,
        xminorticks=false,
        xlabel={$\tau$},
        ylabel={$L^2$-error},
        xmin=8e-4, xmax=0.04,
        ymin=8e-8, ymax=10,
        xmode=log, ymode=log,
        xtick style={color=black},
        ytick style={color=black},
        legend to name=sublinear-Ltwo-error
    ]
    
    \addplot [semithick, darkorange25512714, mark=o, mark size=3, mark options={solid}]
    table {%
    0.032 0.406699590483875
    0.016 0.10577868671153
    0.008 0.0264979162595601
    0.004 0.0066224315573868
    0.002 0.001655560675635
    0.001 0.000415543731248885
    };
    \addlegendentry{$\varepsilon=10^{-2}$}
    \addplot [semithick, forestgreen4416044, mark=+, mark size=3, mark options={solid}]
    table {%
    0.032 1.99870410250843
    0.016 0.941975979726712
    0.008 0.24536914047246
    0.004 0.061497243414085
    0.002 0.0153766564298069
    0.001 0.00384437878211207
    };
    \addlegendentry{$\varepsilon=10^{-3}$}
    \addplot [semithick, crimson2143940, mark=asterisk, mark size=3, mark options={solid}]
    table {%
    0.032 1.99809328819947
    0.016 1.99982073766145
    0.008 1.90477297986864
    0.004 0.620629733960408
    0.002 0.157600487729003
    0.001 0.0394383823353463
    };
    \addlegendentry{$\varepsilon=10^{-4}$}

    \node at (1.7e-3,1.5) {Boris};
    
    \draw (0.016, 6.4e-2) -- (0.002, 1e-3) -- (0.016, 1e-3) -- cycle;
    \draw (0.014, 4e-2) coordinate[label=below:$2$];
    
    \nextgroupplot[
        log basis x={10},
        log basis y={10},
        scaled y ticks=manual:{}{\pgfmathparse{#1}},
        tick align=inside,
        tick pos=left,
        xmajorgrids,
        ymajorgrids,
        xminorticks=false,
        x grid style={darkgray176},
        y grid style={darkgray176},
        xlabel={$\tau$},
        xmin=8e-4, xmax=0.04,
        ymin=8e-8, ymax=10,
        xmode=log,
        ymode=log,
        xtick style={color=black},
        ytick style={color=black},
        yticklabel=\empty
    ]

    \addplot [semithick, darkorange25512714, mark=o, mark size=3, mark options={solid}]
    table {%
    0.032 0.000145645429860006
    0.016 9.5790709736195e-06
    0.008 6.12218337779466e-07
    0.004 1.94857967356175e-07
    0.002 1.94857967356175e-07
    0.001 1.94857967356175e-07
    };
    \addplot [semithick, forestgreen4416044, mark=+, mark size=3, mark options={solid}]
    table {%
    0.032 0.00134995024341988
    0.016 8.44363461353324e-05
    0.008 5.27940801920202e-06
    0.004 3.64088765468318e-07
    0.002 2.06476546236143e-07
    0.001 2.09677923909877e-07
    };
    \addplot [semithick, crimson2143940, mark=asterisk, mark size=3, mark options={solid}]
    table {%
    0.032 0.0135694269136473
    0.016 0.000847741956502617
    0.008 5.29503196798235e-05
    0.004 3.31238434713736e-06
    0.002 2.6904747017127e-07
    0.001 2.17475287904239e-07
    };

    \node at (1.7e-3,1.5) {mRK4};

    \draw (0.032, 1.024e-4) -- (0.008, 4e-7) -- (0.032, 4e-7) -- cycle;
    \draw (0.028, 5e-5) coordinate[label=below:$4$];
    
\end{groupplot}

\node (l1) at ($(group c1r1.south)!0.5!(group c2r1.south) + (0,-.25)$)
      [below, yshift=-2\pgfkeysvalueof{/pgfplots/every axis title shift}]
      {\ref*{sublinear-Ltwo-error}}; 
    
\end{tikzpicture} 
    \vspace{-4mm}
    \caption{
        \(L^2\)-error of a \GWP 
        approximated by the Boris-type method~\eqref{Algo:Boris} (left) and the mRK4 method (right) against a reference \GWP with coefficients approximated by the classical RK4 method with time stepsize $\tau=10^{-4}$. The potentials are given by~\eqref{eq:sublinear-pot-expls} and the initial values by~\eqref{eq:sublinear-IV}. Different values for \(\scp\) are considered.
    }
    \label{fig:sublinear-L2-convg-old}
\end{figure}


In \Cref{fig:sublinear-energy-Boris-RK4} we compare the error between the energy~\eqref{eq:energy_general} and the initial energy (top) of a \GWP between the three methods, Boris-type, mRK4, and standard RK4. We have fixed \(\alpha = 0\) in~\eqref{eq:sublinear-pot-expls} such that we have time-independent potentials and therefore theoretical energy conservation. As end time we chose \(T = 200\) and \(\scp = 10^{-3}\). On the left, we plot the maximal energy error at each time stamp to the initial energy for different time step-sizes \(\tau\). We see a reduction of order two of the error for the Boris-type method and of order four for the mRK4 and standard RK4 methods. Note, however, that the error constant of the standard RK4 method is worse than that of the mRK4 method and the Boris-type method. In the right plot, we plot the energy error against the initial error at each time stamp for a fixed time step-size \(\tau=10^{-1}\). As we see, the error of the Boris-type method oscillates at the same level, while we see for the standard RK4 and mRK4 methods a drift. Note, however, that the drift of the mRK4 method is way smaller than that of the standard RK4 method. Moreover, at the bottom of \Cref{fig:sublinear-energy-Boris-RK4} we illustrate the error of the \(L^2\) norm of a \GWP to norm conservation using time stepsize \(\tau = 10^{-1}\). The error for the Boris-type method and the mRK4 method is close to the machine precision and hence shows conservation of the \(L^2\)-norm, whereby using the standard RK4 method results in a deviation. The \(L^2\)-norm of a \GWP given its parameters can be calculated analytically by~\eqref{eq:L2-norm-GWP}.
\begin{figure}[!htb]
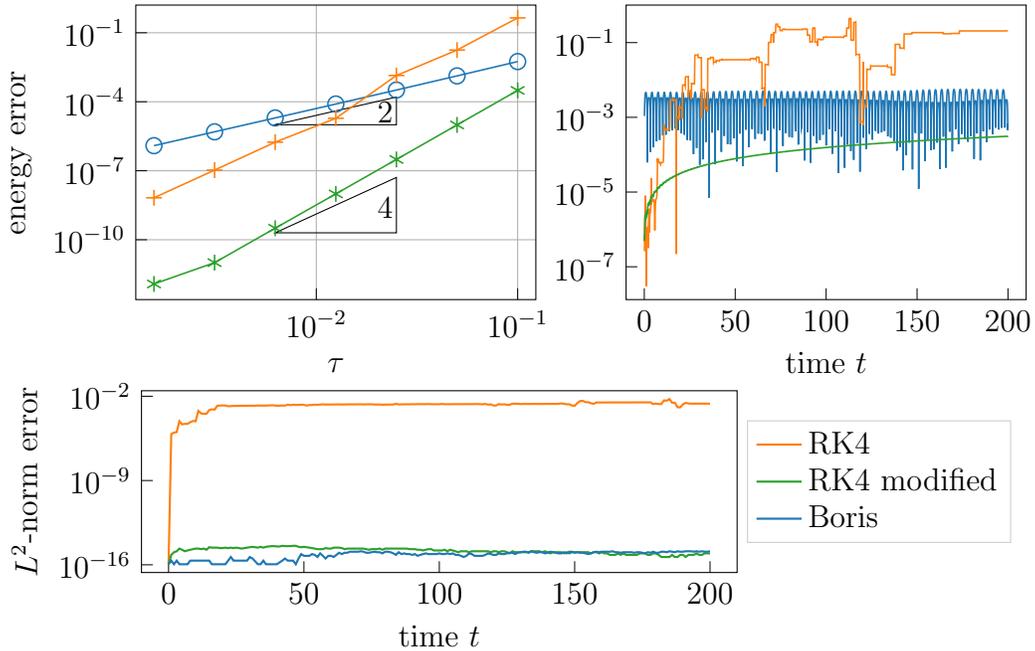

    \centering
    \input{tikz/sublinear/energy-group} \\
    \vspace{-5mm}
    \input{tikz/sublinear/L2-norm-boris-rk4-modified}
    \vspace{-4mm}
    \caption{
        Energy error against the initial energy (top) and $L^2$-norm error (bottom) using the potentials and initial values in Section~\ref{subsec:convg-plots}. The endtime is chosen as \(T=200\) and \(\scp=10^{-3}\). On the left, the maximal energy error is illustrated over all time stamps for different time stepsizes \(\tau\). On the right, the energy error over the time interval \([0, 200]\) is plotted for a fixed time stepsize \(\tau = 10^{-1}\).
    }
    \label{fig:sublinear-energy-Boris-RK4}
\end{figure}

%

\subsection{Penning trap}\label{subsec:Penning-trap}

As a second example, we apply the time-integrators to 
the three--dimensional quantum dynamics of a proton in a hyperbolic Penning trap, see \Cref{subsec:intro-Penning-trap}. 
We consider the Schr\"odinger equation \eqref{eq:TDSE-no-units} with the trap parameters of \Cref{tab:trap-frequencies}
and an initial \GWP with parameters (in dimensionless units)
\begin{equation}\label{eq:penning-IV}
\begin{split}
    {\pos}^0 &= \begin{pmatrix}
        0.133 & 0.133 & 0.258
    \end{pmatrix}^\top, 
    \quad
    {\fcQ}^0 = \mathrm{diag}({\pos}^0), \\
    {\mom}^0 &= \begin{pmatrix}
        0.133 & 7.492 & 3.879
    \end{pmatrix}^\top, 
    \quad
    {\fcP}^0 = \ii\cdot\mathrm{diag}(7.492, 7.492, 3.879), \\
    %
    {\pha}^0 &= 1.009 - 1.84\cdot 10^{-7}\ii .
\end{split}
\end{equation}
 The initial condition is chosen such that the dynamics are coherent in the sense, that the width of the packet does not change over time. The phase parameter has a non-vanishing real part to be aligned with the analytic expressions for the center motion  
that were recently given in \cite[eqs.~(12) and (17)]{BiaB24} (with $\omega_\bot = \Omega/2$).
As previously mentioned, see for example \Cref{sec:accuracy}, the variational approximation is exact in this case, and the evolution does not require averages but only point evaluations, see \Cref{rem:linear}. In this set-up, we see the exact errors of our numerical schemes to integrate the equations of motion~\eqref{eq:euler-lagrange-q-v} and \eqref{eq:Pha-eqs-all}. 

In \Cref{fig:penning-trap-projected-exact}, we showed the exact trajectory of the position center. 
Virtually the same trajectory is obtained by our numerical simulations for $\tau \lesssim 10^{-3}$, as confirmed by the small errors illustrated in \Cref{fig:penning-component-convg}.
%
There, we present the maximal error over all time steps of the parameters against the exact solution for different step-sizes \(\tau\) in the Frobenius norm scaled by the inverse number of the component. As expected, the Boris-type method converges with order two while the RK4 and mRK4 methods converge with order four. 
We are convinced that the plateau at larger time step-sizes occurs since the parameter evolution is mildly oscillatory because the quotients ${\corrCycloFreq}/{\magnetFreq} \approx 113.25$ and ${B}/{B_m} \approx 114.25$
of the potentials in~\eqref{eq:TDSE-no-units} are not small, and since we did not observe such plateaus when setting these quotients close to one. 
Moreover, the error for the imaginary phase \(\ImPha\) using the RK4 method applied to~\eqref{eq:ham-all} is close to the machine precision since the exact solution is constant. In contrast, in the mRK4 method, 
the error of \(\ImPha\) directly relates to the error of \(\fcQ\) by~\eqref{eq:relation_zeta_I_approx} and thus shows order four. 
Finally, in \Cref{fig:penning-energy-Boris-RK4} we compare the energy error of the three methods, which shows a drift for the RK4 methods but not for the Boris-type algorithm.


\begin{figure}[!htb]
    \centering
    \input{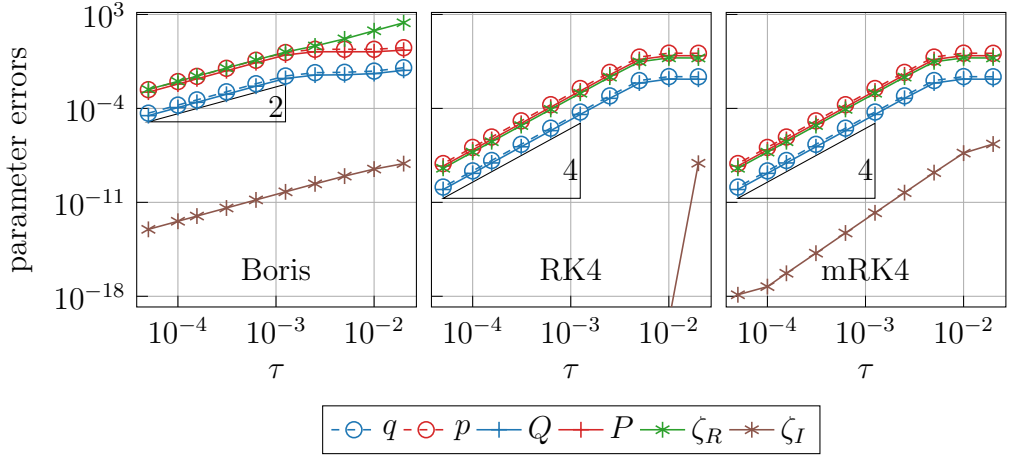}
    \vspace{-4mm}
    \caption{
        Maximum errors of the parameters approximated by the Boris-type method~\eqref{Algo:Boris} (left), the RK4 method (middle) and the mRK4 method (right) against the exact solution using the potentials~\eqref{eq:trap-potentials} with data given in~\Cref{tab:trap-frequencies} and initial values~\eqref{eq:penning-IV}. Measured in the Frobenius norm scaled with the inverse number of components.
    }
    \label{fig:penning-component-convg}
\end{figure}

\begin{figure}[!htb]
    \centering
    \input{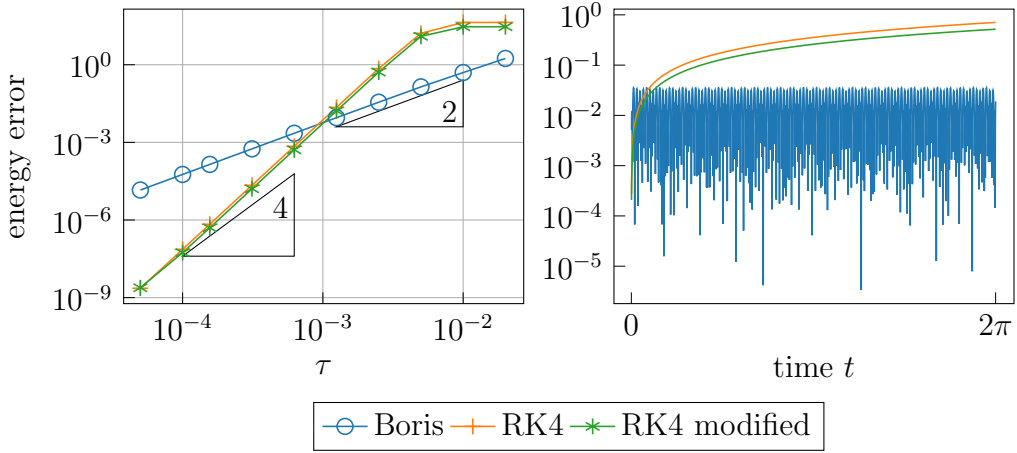}
    \vspace{-8mm}
    \caption{
        Energy error  using the potentials~\eqref{eq:trap-potentials} and initial values~\eqref{eq:penning-IV}. The endtime is chosen as \(T=2\pi\). On the left, the maximal energy error is plotted for different time stepsizes \(\tau\). On the right, 
        we plotted the energy error over time for \(\tau = 2.5\cdot10^{-3}\).
    }
    \label{fig:penning-energy-Boris-RK4}
\end{figure}








\section*{Acknowledgments}
The collaboration among the authors was inspired by discussions during the Workshop ``Nonlinear Optics: Physics, Analysis, and Numerics'' at Mathematisches Forschungsinstitut Oberwolfach. We are thankful to the institute for creating such a pleasant and stimulating atmosphere.
We thank Selina Burkhard for co-supervising the Master thesis of Malik Scheifinger, which was the origin of this project. 
Funded by the Deutsche Forschungsgemeinschaft (DFG, German Research Foundation) -- Project-ID 258734477 -- SFB 1173 and 
-- Project-ID 470903074 -- TRR 352.

\appendix
\section{Appendix: Gaussian calculus and energy formula}\label{appendix}

\begin{lemma}\label{lem:Gauss_calc} Consider smooth functions \(\somePot\colon\Rd\to\bbC^{\dim\times \dim}\) and $w\colon\Rd\to\bbC$, and an arbitrary matrix $\someMatrix\in\bbC^{d\times d}$. 
Let $\vsol\in L^2(\bbR^\dim)$ be a Gaussian wave packet with position $\pos$ and width $\wm = \ReC + \ii\ImC$. Then, 
    \begin{align}
        &\mean{\somePot(x-\pos)} = \frac\scp2\,  \sum_{\ell=1}^\dim \left((\mean{\partial_\ell\somePot}
        \ImCinv)_{k\ell}
        \right)_{k=1}^\dim,\\
        &\mean{ (x-\pos)^\top \somePots\someMatrix (x-\pos) } = \frac\scp2 \mean{\somePots}  \tr(\someMatrix \ImCinv)
        + \frac{\scp^2}4 \tr(\mean{\nabla^2 \somePots} \ImCinv \someMatrix \ImCinv ).
    \end{align}
    In particular, 
    $
    \mean{\somePots(x-\pos)} = \frac\scp2 \ImCinv\mean{\nabla\somePots}
    $.
\end{lemma}

\begin{proof}
The position density 
\[
\abs{\vsol(t,x)}^2 =  \exp\!\left(-\frac1\scp (x-\pos)^\top \ImC (x-\pos) - \frac2\scp \ImPha\right)
\]
satisfies
$
        \nabla\abs{\vsol(t,x)}^2 = -\frac{2}{\scp}\ImC(x-\pos)\abs{\vsol(t,x)}^2.
$
Therefore, partial integration yields that
\begin{align}
\mean{\somePot (x-\pos)} &= -\frac\scp2 \int_{\bbR^\dim} \somePot(x)\, \ImCinv \nabla\abs{\vsol(t,x)}^2 \dx\\
&= \frac\scp2  \sum_{\ell=1}^\dim \left( \mean{ 
 \partial_\ell(\somePot \ImCinv)_{k,\ell}}  \right)_{k=1}^\dim.
 \end{align}
Similarly, two partial integrations imply that 
\begin{align}
    &\mean{\somePots(x-\pos)^\top \someMatrix (x-\pos)} 
    =
    \sum_{k,\ell=1}^\dim \mean[\Big]{ \somePots (\ImC(x-\pos))_k (\ImCinv\someMatrix\ImCinv)_{k\ell} (\ImC(x-\pos))_\ell}\\ 
    &\qquad=
    \frac\scp2 \sum_{k,\ell=1}^\dim \left( 
    \mean{\somePots} (\ImC)_{k\ell} + \mean[\big]{\partial_\ell \somePots (\ImC(x-\pos))_k} \right) (\ImCinv\someMatrix\ImCinv)_{k\ell}\\
    &\qquad=
    \sum_{k,\ell=1}^\dim \left( 
    \frac\scp2 \mean{\somePots} (\ImC)_{k\ell} + \frac{\scp^2}{4} \mean[\big]{\partial_k \partial_\ell \somePots} \right) (\ImCinv\someMatrix\ImCinv)_{k\ell}\\   
    &\qquad=
    \frac\scp2 \mean{\somePots} \,\tr(\someMatrix\ImCinv) 
    + \frac{\scp^2}{4} \tr(\mean{\nabla^2 \somePots}\ImCinv\someMatrix\ImCinv) .
\end{align}
\end{proof}

\begin{lemma}\label{Lemma:energy_equation}
    Let $\vsol$ be a normalized \GWP with center $(\pos,\mom)$ and width $\wm = \ReC+\ii\ImC$. Then it holds for the energy
    \begin{align}
        \mean{H}
        &= 
        \frac12\mom^2 - \mom^\top\mean{\mgPot} 
        + \frac12\mean{\mgPot^2} 
        + \mean{\mltPot} + \frac\scp4 R_H
    \label{eq:energy_general}
    \end{align}
    with remainder
    \[
       R_H = \tr\left(\left(\ReC^2 + \ImC^2 - 2\mean{\jacobian{\mgPot}^\top} \ReC \right)\ImCinv \right). 
    \]
    If the magnetic potential \(\mgPot\) is linear in \(x\), then $\mean{\mgPot} = \mgPot(\pos)$ and 
    \begin{align}
        \mean{\mgPot^2} &= \mgPot(\pos)^2 + \frac\scp2\tr\left(\jacobian{\mgPot}(\pos)^\top\jacobian{\mgPot}(\pos)\ImCinv\right)
        \label{eq:energy_linear}
    \end{align}
\end{lemma}
\begin{proof}
    In the following we ignore the dependence on $t$ to simplify the notation. Since $\mean{H}  = \mean{h} = 
    \frac12\mean{(\xi-\mgPot)^2} + \mean{\mltPot}$, we only need to work on the mean of $\op((\xi-\mgPot)^2) = (\op(\xi-\mgPot))^2$. We have
    \begin{align}
        &\mean{(\xi-\mgPot)^2} = \langle (\xi-\mgPot)\vsol,(\xi-\mgPot)\vsol\rangle \\
        &= \langle (\wm(x-\pos) + (\mom-\mgPot))\vsol, (\wm(x-\pos) + (\mom-\mgPot))\vsol\rangle\\
        &= \mean{(x-\pos)^\top \wm^*\wm (x-\pos)} + \mean{(\mom-\mgPot)^\top (\wm + \wm^*) (x-\pos)} + \mean{(\mom-\mgPot)^2}\\
        &= \frac\scp2 \tr(\wm^*\wm\ImCinv) - \frac\scp2 \tr(\mean{\jacobian{\mgPot}^\top} (\wm+\wm^*)\ImCinv) + \mean{(\mom-\mgPot)^2}
    \end{align}
    due to \Cref{lem:Gauss_calc}.
    For the traces, we have 
    \begin{align}
        &\tr(\wm^*\wm\ImCinv) = \tr((\ReC-\ii\ImC)(\ReC+\ii\ImC)\ImCinv) = \tr((\ReC^2+\ImC^2)\ImCinv)\\
    &\tr(\mean{\jacobian{\mgPot}^\top}(\wm+\wm^*)\ImCinv) = 2\,\tr(\mean{\jacobian{\mgPot}^\top}\ReC\ImCinv).
    \end{align}
    Combining the terms, we obtain \eqref{eq:energy_general}. In the linear case, we expand  $\mgPot(x) = \mgPot(\pos) + \jacobian{\mgPot}(\pos)^\top(x-\pos)$ and use \Cref{lem:Gauss_calc} to prove \eqref{eq:energy_linear}.
\end{proof}

\small 
\bibliographystyle{elsarticle-num}
\bibliography{refs}

\end{document}